\documentclass[12pt,table]{article}
\usepackage{color}
\usepackage[table,xcdraw,dvipsnames]{xcolor}
\usepackage{soul}
\usepackage[hidelinks,colorlinks=true,bookmarks=true,citecolor=Green ,linkcolor=red,urlcolor=blue]{hyperref}
\usepackage{amsmath,scalerel,amssymb}
\usepackage{amsthm}
\usepackage{mathrsfs}
\usepackage{cite}
\usepackage{MnSymbol}
\usepackage[mathletters]{ucs}

\usepackage[utf8]{inputenc} 

\usepackage{bookmark}
\usepackage{stackengine}
\usepackage{hypcap}
\usepackage[]{graphicx} 
\usepackage[percent]{overpic}
\usepackage{wrapfig}
\usepackage{float} 
\usepackage{rotating}
\usepackage{tikz}

\usepackage{accents}

\usepackage{array}
\newcolumntype{M}[1]{>{\centering\arraybackslash}m{#1}}
\newcolumntype{N}{@{}m{0pt}@{}}
\usepackage{multicol}
\usepackage{multirow}
\usepackage{caption}
\usepackage{subcaption}
\usepackage[left=2cm,right=2cm,top=2cm,bottom=2cm]{geometry}
\usepackage{setspace}
\usepackage[most]{tcolorbox}
\usepackage{pifont}
\usepackage{wasysym}

\numberwithin{lemma}{section}

\numberwithin{proposition}{section}

\numberwithin{definition}{section}

\numberwithin{example}{section}
\newtheorem{theorem}{Theorem}
\numberwithin{theorem}{section}

\numberwithin{corollary}{section}

\newtheorem{remark}{Remark}
\numberwithin{remark}{section}

\newcommand{\beqn}{\begin{equation} \begin{aligned}}
\newcommand{\eeqn}{\end{aligned}\end{equation}}
\newcommand{\beqnn}{\begin{equation*} \begin{aligned}}
\newcommand{\eeqnn}{\end{aligned}\end{equation*}}

\usepackage{environ}

\NewEnviron{neweq}{%
	\begingroup
	\allowdisplaybreaks
	\begin{align}
	\begin{split}
	\BODY
	\end{split}
	\end{align}
	\endgroup}

\NewEnviron{neweq_non}{%
	\begingroup
	\allowdisplaybreaks
	\begin{align*}
	\BODY
	\end{align*}
	\endgroup}
\soulregister\cite7
\soulregister\ref7
\soulregister\pageref7
\colorlet{Mycolor1}{green!10!orange!90!}
\colorlet{Mycolor2}{Yellow!80!}
\colorlet{Mycolor3}{green!50!}
\colorlet{Mycolor4}{Red!90!}

\makeatletter
\makeatletter
\renewcommand*{\@fnsymbol}[1]{\ifcase#1\or *\or $\star$\or$\dagger$\or$\ddagger$\or \else\@arabic{#1}\fi}
\makeatother

\date{}

\title{M-current Induced Bogdanov-Takens Bifurcation and Switching of Neuron Excitability Class\footnote{This research is supported in part by the Natural Sciences and Engineering Research Council of Canada.}}

\makeatletter
\newcommand{\specificthanks}[1]{\@fnsymbol{#1}}
\makeatother

\author{Isam Al-Darabsah\footnote{Department of Applied Mathematics, University of Waterloo, Waterloo, ON,  N2L 3G1, Canada.} \textsuperscript{,}\footnote{Email: ialdarabsah@uwaterloo.ca} \and Sue Ann Campbell\textsuperscript{\specificthanks{2}}\textsuperscript{,}\footnote{Email: sacampbell@uwaterloo.ca} }

\begin{document}

\maketitle

\begin{abstract} 
In this work, we consider a general conductance-based neuron model with the inclusion of the acetycholine sensitive, M-current. We study bifurcations in the parameter space consisting of the applied current, $I_{app}$ the maximal conductance of the M-current, $g_M$, and the conductance of the leak current, $g_L$.
We give precise conditions for the model that ensure the existence of a Bogdanov-Takens (BT) point and   show such a point can occur by varying $I_{app}$ and $g_{M}$.
We discuss the case when the BT point becomes a Bogdanov-Takens-Cusp (BTC) point and show that such a point can occur in the three dimensional parameter space.    
The results of the bifurcation analysis are applied to different neuronal models and are verified and supplemented by numerical bifurcation diagrams generated using the
package \textsf{MATCONT}. 
We conclude that there is a transition in the neuronal excitability type organized by the  BT   point and the neuron switches from Class-I to Class-II as conductance of the M-current increases.
\end{abstract}


\section{Introduction}

Neuromodulators are chemicals released by neurons that can alter
the behaviour of individual neurons and large populations of neurons. 
Examples include dopamine, seratonin and acetylcholine.
These chemicals occur widely in the brain and can affect many
types of neurons. The effect of neuromodulators ranges from altering
the membrane properties of individual neurons to altering synaptic
transmission.
 
The M-current is a voltage dependent, noninactivating potassium current, 
which has been shown to occur in many neural types including 
excitatory neurons in the cortex \cite{stiefel2008cholinergic} and inhibitory
neurons in the hippocampus \cite{lawrence2006somatodendritic}. 
Its name arises from the fact that this current is in down-regulated by 
the presence of the neuromodulator acetylcholine through its action on
the muscarinic receptor. At the simplest level, this current reduces firing activity since it is a potassium current \cite{mccormick1993actions,lawrence2006somatodendritic}.
However, this current has been implicated in many aspects of both individual 
cell and network activity.

 Before reviewing the literature on M-current we first recall some terminology. Neurons and neural models are often classified by their membrane excitability class as described by Hodgkin \cite{hodgkin1948} neurons with the Class-I excitability have a continuous 
frequency-current (F/I) curve because they begin repetitive firing with
zero frequency from the resting state.
On the other hand, the frequency-current curve of Class-II neurons is discontinuous because they start firing with non-zero frequency from the resting state 
\cite{izhikevich2007dynamical}.
The phase resetting curve (PRC) describes the effect of stimulation on the phase of an oscillator as a function of the phase at which the stimulus is delivered. At phases where the PRC is positive the phase is advanced, meaning the period of the oscillator is increased by the perturbation. At phases where the PRC is negative the phase is delayed, corresponding to a decrease in the period of the oscillator \cite{canavier2006phase,ermentrout2010mathematical}.  As introduced by Hansel~\cite{hansel1995synchrony} a Type-I PRC is one where an excitatory stimulus produces only phase advances, while in a Type-II PRC either phase advance or phase delay can occur, depending on the phase of the stimulus. For two oscillators with reciprocal excitatory coupling, a Type-I PRC means the coupling cannot synchronize the oscillators, while a Type-I PRC means that the coupling can synchronize the oscillators. For inhibitory coupling the opposite occurs \cite{canavier2006phase,ermentrout2010mathematical}.
 Another important classification of neurons is whether or not they exhibit subthreshold oscillations. Neurons that do exhibit subthreshold oscillations are called resonators, while neurons that do not are called integrators.
 
At the single cell level, the M-current has been shown to affect the neuronal 
excitability \cite{stiefel2008cholinergic,brown2009neural} and resonant properties \cite{hu2002two,prescott2006nonlinear,prescott2008pyramidal}. 
For example,  in  \cite{stiefel2008cholinergic}, the authors recorded from layer II/III pyramidal neurons and determined PRCs. 
Stiefel et al.~\cite{stiefel2008cholinergic} found that down-regulation of slow voltage-dependent potassium currents such as the M-current can switch the PRC from Type-II to Type-I, thus changing the expected synchronization of pairs of coupled neurons.
In a follow-up paper \cite{stiefel2008cholinergic}, they showed for that the M-current could produce the same effect in several
different neural models. The work of \cite{fink2011cellularly} showed that these differences in
PRC type due to M-current modulation translate into differences in synchronization properties 
in networks of model neurons.
The experimental work of \cite{prescott2006nonlinear,prescott2008pyramidal} showed that increased membrane conductance (shunting) could switch a hippocampal pyramidal neuron from an integrator to a resonator. Using a simple model, they attributed this change to the combined effect of shunting (modelled as a leak current) and the M-current. Interpreting the shunt as representing the effect of background synaptic on a neuron, Prescott et al.~\cite{prescott2008pyramidal} concluded that neurons that present as integrators in vitro may act as resonators in vivo.
At the network level the M-current has also been implicated in 
the organization of rhythms in striatal microcircuits. 
In \cite{zhou2018m},  the authors studied an inhibitory neuron model with M-current under forcing from gamma pulses and a sinusoidal current of theta frequency. They found that the M-current expands the phase-locking frequency range of the network, counteracts the slow theta forcing and admits bistability in some parameter range.
In \cite{chartove2020biophysical}, the effects of the M-correct on $\beta$ oscillations  was studied.

In all the studies cited above, the effect of acetylcholine, through
the M-current, was explored in models for specific cells. While this
is important for understanding the behaviour of specific cells and
brain networks, it can be difficult to extract the essential effects of 
the M-current from its interplay with other specific currents in the
models. 
Here we take a different approach and consider the effect
of the M-current in a general conductance-based model. We study
the bifurcations of the model in the parameter space of two parameters
common to any conductance-based model with an M-current: the
applied current and the maximal conductance of the M-current.
We derive necessary and sufficient conditions for the existence of 
two codimension-two bifurcations of the resting equilibrium point:
the Bogdanov-Takens (BT) bifurcation and the Cusp (CP) bifurcation.

The Bogdanov-Takens (BT) bifurcation is associated with 
an equilibrium point that has a zero eigenvalue with algebraic multiplicity 
two and geometric multiplicity one. 
The Cusp bifurcation occurs when three equilibrium points coalesce into one,
and can be thought of as the simultaneous occurrence of two fold bifurcations.
When an equilibrium point simultaneously undergoes a BT and  Cusp bifurcation, 
a Bogdanov-Takens-Cusp (BTC) occurs, which is a codimension three bifurcation.
We show that variation of a third parameter, the leak conductance, can 
lead to a BTC bifurcation point.

In the literature, there are many instances where the presence of 
the BT, Cusp and to a lesser extent the BTC, bifurcations have been shown
to occur in particular conductance-based models. For example, the presence 
of BT and Cusp bifurcations \cite{guckenheimer1993bifurcation} and BTC bifurcation \cite{mohieddine2008chaos} have been shown in the Hodgkin-Huxley model.  In \cite{tsumoto2006bifurcations}, the author showed the existence of BT and Cusp 
bifurcations in Morris-Lecar model \cite{Mor_Lec}.
While in \cite{ermentrout2010mathematical} the BT and Cusp bifurcations 
were shown in the Wang-Buzs\'aki interneuron model \cite{wang1996gamma}. 
The majority of these studies used numerical bifurcation analysis to show 
that these bifurcations
occur as particular parameters are varied with all other parameters fixed
at some specific, biologically relevant values. The prevalence of these
codimension two bifurcations in particular studies, would seem to indicate
that these bifurcations are associated with some underlying structure in 
conductance-based models in general. Indeed, two recent papers give
support to this hypothesis.
The authors in \cite{kirst2015fundamental} considered a general 
conductance-based neuron model and studied the existence of the BTC point 
in the parameter space of the applied current, leak conductance, and 
capacitance.  In \cite{pereira2015bogdanov}, the authors give general 
conditions the existence of the BT bifurcation in any conductance-based
model. Our work builds on these latter two papers and extends them to
the situation where an M-current is present in the model. 

To understand the implications of the co-dimension two bifurcations, we related them to the neural behaviours
described described above. The resonance property of neural models quite simply related to the bifurcation that causes the loss of stability of the resting state when the applied current is increased. If this bifurcation is a Hopf bifurcation the model is a resonator, otherwise it is an integrator. As pointed out by Izhikevich, a Bogdanov-Takens bifurcation can switch the resonator type of a neuron \cite{izhikevich2007dynamical}.
Class I/II excitability was first linked to bifurcations in neural models
by Rinzel and Ermentrout \cite{rinzel1998analysis}.
Rinzel and Ermentrout showed that neuronal models where the  onset of repetitive firing occurs via a saddle-node bifurcation on an invariant circle are Class-I,
while models where the onset occurs via a subcritical Andronov-Hopf bifurcations
 are Class-II. This link can be extended to other types of bifurcations by studying the associated F/I curves.
The excitability class of individual neurons has been linked to the 
synchronization properties of the neuron in a network through the phase resetting curve (PRC). In particular,
it has been shown in certain circumstances that Class-I neurons have
Type-I PRCs \cite{ermentrout1996type}. No conclusive link between Class-II neurons and a particular PRC Type was found in that paper.
More recently, Izhikevich has made a subtly different classification of excitability, based on ramped current inputs as opposed to step current inputs. Izhikevich defines Class I/II excitability based on the bifurcation that causes the loss of stability of the resting state when the current is increased. 
Further, Izhikevich defines Class I/II spiking by the bifurcation that destroys the stable oscillations as the current is decreased \cite{izhikevich2007dynamical}.
A focus for this paper will be on how the presence of a BT point is linked to the emergence
of a Hopf bifurcation and thus could be associated with a change of
oscillation class for a conductance-based neural model.

The paper is organized as follows. 
In the next section, we provide a general conductance-based neuron model
with the inclusion of the $M-$current  and study the existence of the steady-state solutions.
In Sections \ref{sec_BT} and \ref{sec:exist}, we give a complete characterization of the BT bifurcation, provide a condition for the Cusp bifurcation and discuss the existence of Bogdanov-Takens-Cusp (BTC) bifurcation.
In Section \ref{sec_Numerical}, we consider three example models, and
show all three models exhibit the BT, CP and BTC bifurcation points.
We construct bifurcation diagrams using \textsf{MATCONT} to explain possible behaviour of each example and use the numerical solution of each model 
to construct the frequency-current curves.  In Section \ref{sec:Implications}, we use numerical simulations 
to study the influence of varying of $g_M$ on the neurons synchronization in 
two coupled neurons model with synaptic coupling.
In Section \ref{sec_Conclusions}, we discuss our results.

\section{General model}
\label{sec_model}
In nondimensional variables, a  general conductance-based neuron model
with the inclusion of the $M-$current can be written as:
   \beqn 
   \label{Model}
{C_m}\frac{{dV}}{{dt}} &= {I_{app}} - {g_L}\left( {V - {V_L}} \right) - {g_M}w\left( {V - {V_K}} \right) + {I_{ion}}(V,a)\\
\frac{{dw}}{{dt}}& = \frac{1}{{{r}(V)}}\left( {{w_\infty }(V) - w} \right)\\
\frac{{d{a}}}{{dt}}& ={\tau}^{-1}(V)\left( {{a_{\infty }}(V) - {a}} \right)
   \eeqn
where $a = {\big({a_3}, \ldots ,{a_N}\big)^T}$,
\[{a_\infty }(V) = {\big({a_{3,\infty }}(V), \ldots ,{a_{N,\infty }}(V)\big)^T},\quad \tau^{-1} (V) = diag\bigg(\frac{1}{{\tau _3}(V)}, \ldots ,\frac{1}{{\tau _N}(V)}\bigg)\]
and
\[{I_{ion}}(V,a) = \sum\limits_{i = 3}^N {{g_i}\left( {{V_i} - V} \right)\prod\limits_{j \in {\phi_i}} {a_j^{{p_j}}} } \]
where ${I_{app}}$ is the applied current and $\phi_i$ is the set of indexes that
represents the identities of the gating variables present in a given ionic current. 
In the rest of the manuscript, we assume that all conductances $g_j$ are positive, and  the steady state activations, $w_{\infty}$ and $a_{j,\infty}$, $j=3,\ldots,N$, are non-negative bounded functions ($ 0\le f(V)\le 1$), monotonic, $C^3(\mathbb{R},\mathbb{R})$, and
become sufficiently flat in the limits $V\to \pm\infty$.

\subsection{Equilibria.} 
By applying the scaling $t\to\frac{t}{C_m}$, system (\ref{Model}) can be written as 
\beqn
   \label{Model_V2}
\frac{{dV}}{{dt}} &= {I_{app}} - {g_L}(V-{V_L}) - {g_M}w\left( {V - V_K} \right) + {I_{ion}}(V,a):=f_1(V,w,a)\\
\frac{{dw}}{{dt}} &= \frac{C_M}{{r(V)}}\left( {{w_\infty }(V) - w} \right):=f_2(V,w)\\
\frac{{da}}{{dt}} &= C_M {\tau^{-1} (V)}\left( {{a_\infty }(V) - a} \right):=f_3(V,a)
 \eeqn 
where  $f_3(V,a)=\big(f_{33}(V,a_3),\ldots, f_{3N}(V,a_N)\big)^T$. 
Assume that (\ref{Model_V2}) has an equilibrium point $E^{^*}=(V^*,w^*,a_0^*)$.  From the equations above it follows that
\[{w_\infty }({V^*}) = {w^*}\quad\text{and}\quad{a_\infty }({V^*}) = {a^*}\]
where $V^*$ satisfies 
\begin{neweq}
I_{\infty}(V^*)=0.
\label{eq:ustar}
\end{neweq}
Here, $I_{\infty}$ is the 
steady-state $I-V$ curve \cite{izhikevich2007dynamical,rinzel1998analysis} defined by
\begin{equation}
 I_{\infty}(V)={I_{app}} - {g_L}(V-{V_L}) - {g_M}w_{\infty}(V)\left( {V - V_K} \right) + {I_{ion,\infty}}(V)
    \label{eq:I_infty}
\end{equation}
where ${I_{ion,\infty }}(V) = {I_{ion}}(V,{a_\infty }(V))$ is the stationary ionic current.
Notice that  (\ref{eq:ustar}) can be written as
\[{I_{app}}  =  {g_L}({V^*}-V_L)+{g_M}{w_\infty }(V^*)\left( {{V^*} - V_K} \right) - {I_{ion}}(V^*,{a_\infty }({V^*})) :=U(V^*).\]
Now, we write $U(V^*)$ in the form
\[U(V^*)=\left( {{g_L} + {g_M}{w_\infty }(V^*) + {h_2}(V^*)} \right)V^*-\left( {  {g_M}{m_\infty }(V^*) + {h_1}(V^*)} \right)- {g_L}{V_L} \]
where $h_1$ and $h_2$ are polynomials in the variables $a_{j,\infty}(V)$, and hence, 
\[\mathop {\lim }\limits_{V^* \to  \pm \infty } U(V^*) =  \pm \infty
\]
because  all maximal conductances and activation variables are
positive and bounded. 
Thus, equation (\ref{eq:ustar}) has at least one solution.

\section{Bogdanov-Takens bifurcation} 
\label{sec_BT}
In the following we  discuss Bogdanov-Takens point (BT point) of codimension $2$ in $(I_{app}, g_M)$-plane, when all other parameters in the model are fixed.

 Assume $V^*$ is a solution of (\ref{eq:ustar}), then there exist parameters $(I_{app}^*,g^*_M)$ such that
\begin{equation}\label{eq:I_star}
    {I_{app}^*} = {g_L}({V^*}-V_L) + {g^*_M}{w_\infty }({V^*})\left( {{V^*} - V_K} \right) - {I_{ion}}(V^*,{a_\infty }({V^*})).
\end{equation}
It is well known \cite{Guck_Holmes,kuznetsov2005practical,kuznetsov2013elements} that the equilibrium point $V^*$ is BT  point if the zero eigenvalue has algebraic multiplicity two and geometric multiplicity one. Using an approach similar to \cite{kirst2015fundamental,pereira2015bogdanov}, we obtain the following.

\begin{theorem}\label{Th:BTpoint}
Let $V^*$ be a solution of (\ref{eq:ustar}) at $(I_{app}^*,g^*_M)$ and define
\[\partial _a^{{f_1}} = {\left( {\frac{{\partial {f_1}}}{{\partial {a_3}}} \ldots ,\frac{{\partial {f_1}}}{{\partial {a_N}}}} \right)^T}, \quad \partial _V^{{f_3}} = {\left( {\frac{{\partial {f_{33}}}}{{\partial V}} \ldots ,\frac{{\partial {f_{3N}}}}{{\partial V}}} \right)^T}.\]
 Assume 
\begingroup
\allowdisplaybreaks
\begin{align}
{\left. {\frac{d}{{dV}}{I_\infty }(V)} \right|_{{V^*}}} &= 0\label{BT:Con_A1}\\
1 +\frac{r^2}{C_M^2}{\left. {\frac{{\partial {f_1}}}{{\partial w}}} \right|_{{E^*}}}{\left. {\frac{{\partial {f_2}}}{{\partial V}}} \right|_{{E^*}}} + \frac{1}{C_M^2}\left( {{{\left. {\partial _a^{f_1^T}} \right|}_{{E^*}}}} \right){\tau ^2}\left( {{{\left. {\partial _V^{{f_3}}} \right|}_{{E^*}}}} \right) &= 0.\label{BT:Con_A2}
\end{align}
\endgroup
Then $E^*$ is  an ordinary BT  point of codimension $2$.
\end{theorem}
\begin{proof}
Let $F=(f_1,f_2,f_3)^T$. Then, the Jacobian of (\ref{Model_V2}) is 
 \beqnn
 \arraycolsep=5.3pt\def\arraystretch{1.7}
DF(V,w,a)  = \left( {\begin{array}{*{20}{c}}
{\frac{{\partial {f_1}}}{{\partial V}}}&{\frac{{\partial {f_1}}}{{\partial w}}}&{\partial {{_a^{{f_1^T}}}}}\\
{\frac{{\partial {f_2}}}{{\partial V}}}&{ -C_M {r^{ - 1}}}&0\\
{\partial _V^{{f_3}}}&0&{ -C_M {\tau ^{ - 1}}}
\end{array}} \right)
 \eeqnn
  where $r^{-1}=\frac{1}{r(V)}$, $\tau^{-1} = diag\big(\frac{1}{{\tau _3}(V)}, \ldots ,\frac{1}{{\tau _N}(V)}\big).$

When ${B_1} \in {\mathbb{R}^{n \times n}},{B_2} \in {\mathbb{R}^{n \times m}},{B_3} \in {\mathbb{R}^{m \times n}},{B_4} \in {\mathbb{R}^{m \times m}}$, we have (see \cite{bernstein2009matrix})
\[\det \left( {\begin{array}{*{20}{c}}
{{B_1}}&{{B_2}}\\
{{B_3}}&{{B_4}}
\end{array}} \right) = \left( {\det {B_4}} \right)\det \left( {{B_1} - {B_2}B_4^{ - 1}{B_3}} \right).\]
\sloppy Let  $A=DF(V^*,m^*,a^*)$. Then, by taking ${B_1} = \left( {\frac{{\partial {f_1}}}{{\partial V}}} -\lambda \right)$, ${B_2} = \left( {\begin{array}{*{20}{c}}
{\frac{{\partial {f_1}}}{{\partial w}}}&{\partial {{_a^{{f_1}}}^T}}
\end{array}} \right)$, ${B_3} = \left( {\begin{array}{*{20}{c}}
{\frac{{\partial {f_2}}}{{\partial V}}}&{\partial _V^{{f_3}}}
\end{array}} \right)^T$ and ${B_4} = diag\left( { -C_M {r^{ - 1}}-\lambda, -C_M {\tau ^{ - 1}}-\lambda I} \right)$,
we have 
\[\det(A-\lambda I):=\Delta(\lambda)=\Delta_1(\lambda)\Delta_2(\lambda)\]
where
\[{\Delta _1}(\lambda ) = {( - 1)^{N - 1}}\left( {\lambda  + C_M{r^{ - 1}}} \right)\prod\limits_{j = 3}^N {\left( {\lambda  +C_M \tau _j^{ - 1}} \right)} \]
and
\[{\Delta _2}(\lambda ) = \frac{{\partial {f_1}}}{{\partial V}} - \lambda  + {\left( {\lambda  +C_M {r^{ - 1}}} \right)^{ - 1}}\frac{{\partial {f_1}}}{{\partial w}}\frac{{\partial {f_2}}}{{\partial V}} + \partial _a^{f_1^T}{\left( {\lambda I +C_M {\tau ^{ - 1}}} \right)^{ - 1}}\partial _V^{{f_3}}.\]
Consequently, we have
\[\Delta (0) = {\Delta _1}(0){\Delta _2}(0) = {\Delta _1}(0)\left( {\frac{{\partial {f_1}}}{{\partial V}} +C_M r\frac{{\partial {f_1}}}{{\partial w}}\frac{{\partial {f_2}}}{{\partial V}} +C_M \partial _a^{f_1^T}\tau \partial _V^{{f_3}}} \right).\]
Notice that
\begin{neweq}\label{Eq:deriv}
{\left. {\frac{{\partial {f_2}}}{{\partial V}}} \right|_{{E^*}}} = C_M{r^{ - 1}}(V^*)\left( {{{\left. {\frac{d}{{dV}}{w_\infty }(V)} \right|}_{{V^*}}}} \right),\qquad {\left. {\partial _V^{{f_3}}} \right|_{{E^*}}} =C_M {\tau ^{ - 1}}(V^*){\left. {\partial _V^{{a_\infty }}} \right|_{{V^*}}}.
\end{neweq}
Thus, at $E^*$, we have that the equation 
\begin{align}
\frac{{\partial {f_1}}}{{\partial V}} +\frac{r}{C_M}\frac{{\partial {f_1}}}{{\partial w}}\frac{{\partial {f_2}}}{{\partial V}} +\frac{1}{C_M} \partial {_a^{{f_1^T}}}\tau \partial _V^{{f_3}} = 0\label{BT:Con_A1_alther} 
\end{align}
is equivalent to ${\left. {\frac{d}{{dV}}{I_\infty }(V)} \right|_{{V^*}}} = 0$. Thus, $\Delta(0)=0$ when (\ref{BT:Con_A1}) holds.

It easy to check that  
\beqnn
\Delta '(0) &= {\Delta _1}(0)\Delta' _2(0) + \Delta' _1(0){\Delta _2}(0)\\
 &=  - {\Delta _1}(0)\left( {1 + \frac{r^2}{C_M^2}\frac{{\partial {f_1}}}{{\partial w}}\frac{{\partial {f_2}}}{{\partial V}} + \frac{1}{C_M^2}\partial _a^{f_1^T}{\tau ^2}\partial _V^{{f_3}}} \right)\\
 &~~~+ \Delta' _1(0)\left( {\frac{{\partial {f_1}}}{{\partial V}} +C_M r\frac{{\partial {f_1}}}{{\partial w}}\frac{{\partial {f_2}}}{{\partial V}} +C_M \partial _a^{f_1^T}\tau \partial _V^{{f_3}}} \right).
\eeqnn
Thus,  at $E^*$, $\Delta '(0)=0$ when (\ref{BT:Con_A1}) and (\ref{BT:Con_A2}) hold. Hence, $\lambda=0$ is a double root.

Now, we show that a Jordan block arises when  $\lambda=0$ is a double multiplicity root. In other words, when (\ref{BT:Con_A1}) and (\ref{BT:Con_A2}) hold,  we demand the existence of four generalized eigenvectors $q_0,q_1,p_0,p_1$ of $A$ such that
\[ Aq_0=0,\quad  Aq_1=q_0,\quad  A^Tp_1=0,\quad  A^Tp_0=p_1.\]
Let $q_i=(q_{i1},\ldots,q_{iN})^T$ and $p_i=(p_{i1},\ldots,p_{iN})^T$ for $i\in\{0,1\}$. Then, we obtain from $Aq_0=0$ the following equations
\begin{align}
{q_{01}}\frac{{\partial {f_1}}}{{\partial V}} + {q_{02}}\frac{{\partial {f_1}}}{{\partial w}} + \partial {_a^{{f_1^T}}}{\left( {{q_{03}} \ldots ,{q_{0N}}} \right)^T} &= 0\label{Aq0:1}\\
{q_{01}}\frac{{\partial {f_2}}}{{\partial V}} - {q_{02}}C_M{r^{ - 1}} &= 0\label{Aq0:2}\\
{q_{01}}\frac{{\partial {f_{3j}}}}{{\partial V}} - \tau _j^{ - 1}{q_{0j}} &= 0,\qquad j = 3, \ldots ,N.\label{Aq0:3}
\end{align}
From (\ref{Aq0:2}) and (\ref{Aq0:3}), we have 
\[{q_{02}} = {q_{01}}\frac{r}{C_M}\frac{{\partial {f_2}}}{{\partial V}} \quad\text{and}\quad {q_{0j}} = {q_{01}}\frac{\tau _j}{C_M}\frac{{\partial {f_{3j}}}}{{\partial V}},j = 3, \ldots ,N\]
respectively. Hence, 
\[\arraycolsep=1.4pt\def\arraystretch{1.7}
{q_0} = {q_{01}}\left( {\begin{array}{*{20}{c}}
1\\
{\frac{r}{C_M}\frac{{\partial {f_2}}}{{\partial V}}}\\
{\frac{\tau}{C_M} \partial _V^{{f_3}}}
\end{array}} \right)\]
and 
it follows from (\ref{Aq0:1}) that
\begin{equation}
    {q_{01}}\left(\frac{{\partial {f_1}}}{{\partial V}} +\frac{r}{C_M}\frac{{\partial {f_1}}}{{\partial w}}\frac{{\partial {f_2}}}{{\partial V}} +\frac{1}{C_M} \partial {_a^{{f_1^T}}}\tau \partial _V^{{f_3}}\right) = 0.
    \label{BT:eq1}
\end{equation}
Similarly, from $Aq_0=q_1$, $A^Tp_1=0$ and $A^Tp_1=p_0$, we have 
\begin{align*}
  \arraycolsep=1.4pt\def\arraystretch{1.7}
{q_1} &= \left( {\begin{array}{*{20}{c}}
{{q_{11}}}\\
{({q_{11}} - {q_{01}}\frac{r}{C_M})\frac{r}{C_M}\frac{{\partial {f_2}}}{{\partial V}}}\\
{({q_{11}}{I_{N - 3}} - {q_{01}}\frac{\tau}{C_M} )\frac{\tau}{C_M} \partial _V^{{f_3}}}
\end{array}} \right),\quad{p_1} = {p_{11}}\left( {\begin{array}{*{20}{c}}
1\\
{\frac{r}{C_M}\frac{{\partial {f_1}}}{{\partial w}}}\\
{\frac{\tau}{C_M} \partial _a^{{f_1}}}
\end{array}} \right),\\
 {p_0} &= \left( {\begin{array}{*{20}{c}}
{{p_{01}}}\\
{({p_{01}} - {p_{11}}\frac{r}{C_M})\frac{r}{C_M}\frac{{\partial {f_1}}}{{\partial w}}}\\
{({p_{01}}{I_{N - 3}} - {p_{11}}\frac{\tau}{C_M} )\frac{\tau}{C_M} \partial _a^{{f_1}}}
\end{array}} \right),  
\end{align*}
where $I_{N-3}$ is the identity matrix of size $N-3$, and 
 \begingroup
\allowdisplaybreaks
\begin{align}
 {q_{11}}\left( {\frac{{\partial {f_1}}}{{\partial V}} + \frac{r}{C_M}\frac{{\partial {f_1}}}{{\partial w}}\frac{{\partial {f_2}}}{{\partial V}} +\frac{1}{C_M} \partial _a^{f_1^T}\tau \partial _V^{{f_3}}} \right)~&\nonumber\\
 - {q_{01}}\left( {1 + \frac{r^2}{C_M^2}\frac{{\partial {f_1}}}{{\partial w}}\frac{{\partial {f_2}}}{{\partial V}} + \frac{1}{C_M^2}\partial _a^{f_1^T}{\tau ^2}\partial _V^{{f_3}}} \right) &= 0\label{BT:eq2}\\
{p_{11}}\left( {\frac{{\partial {f_1}}}{{\partial V}} + \frac{r}{C_M}\frac{{\partial {f_1}}}{{\partial w}}\frac{{\partial {f_2}}}{{\partial V}} +\frac{1}{C_M} \partial _V^{f_3^T}\tau \partial _a^{{f_1}}} \right) &= 0\label{BT:eq3}\\
{p_{01}}\left( {\frac{{\partial {f_1}}}{{\partial V}} + \frac{r}{C_M}\frac{{\partial {f_1}}}{{\partial w}}\frac{{\partial {f_2}}}{{\partial V}} + \frac{1}{C_M}\partial _V^{f_3^T}\tau \partial _a^{{f_1}}} \right)~&\nonumber\\
- {p_{11}}\left( {1 + \frac{r^2}{C_M^2}\frac{{\partial {f_1}}}{{\partial w}}\frac{{\partial {f_2}}}{{\partial V}} +\frac{1}{C_M^2} \partial _V^{f_3^T}{\tau ^2}\partial _a^{{f_1}}} \right) &= 0\label{BT:eq4}.
\end{align}
\endgroup
As the generalized eigenvectors must be non-zero, we let   $q_{01}$ and $p_{11}$  to be nonzero arbitrary constants. Thus, 
when (\ref{BT:Con_A1}) and  (\ref{BT:Con_A2}) hold, equations (\ref{BT:eq1})-(\ref{BT:eq4}) hold. Thus, four generalized eigenvectors exits. 
Hence, $V^*$ is an ordinary Bogdanov-Takens point. 
\end{proof}

\begin{remark}
With the additional condition 
\[p_i^T{q_j} = \left\{ {\begin{array}{*{20}{c}}
1&{\text{if }i = j}\\
0&{\text{if }i \ne j}
\end{array}} \right.,\]
we can guarantees the uniqueness of the generalized eigenvectors $q_0,q_1,p_0,p_1$ of $A$.
\end{remark}

When $V^*$ is a  BT point, system \eqref{Model_V2} has a two-dimensional centre manifold, with normal form given by 
(see e.g., \cite{bognadov1975versal,takens1974singularities,Guck_Holmes,kuznetsov2005practical,kuznetsov2013elements}):
\begingroup
\allowdisplaybreaks
\begin{align}
\begin{split}\label{normal_form_1}
\frac{{d{\xi _0}}}{{dt}} &= {\xi _1},\\
\frac{{d{\xi _1}}}{{dt}} &= {\alpha _2}\xi _0^2 + {\beta _2}{\xi _0}{\xi _1} + O\left( {{{\left\| {\left( {{\xi _0},{\xi _1}} \right)} \right\|}^3}} \right),
\end{split}
\end{align}
\endgroup
where 
\begingroup
\allowdisplaybreaks
\begin{align}\label{eq:alphabeta}
\begin{split}
{\alpha _2} &= \frac{1}{2}p_1^TG({q_0},{q_0}),\\
{\beta _2} &= p_1^TG({q_0},{q_1}) - p_1^T{h_{20}},
\end{split}
\end{align}
\endgroup
where $h_{20}$ is the solution of the equation 
\begin{equation}\label{eq:h20}
    A{h_{20}} = 2{\alpha _2}{q_1} - G({q_0},{q_0})
\end{equation}
and the function $G$ is defined as 
\[\arraycolsep=1.4pt\def\arraystretch{1.7}
G({z_1},{z_2}) := \left( {\begin{array}{*{20}{c}}
{z_1^T{D^2}{f_1}({V^*}){z_2}}\\
{z_1^T{D^2}{f_2}({V^*}){z_2}}\\
{z_1^T{D^2}{f_3}({V^*}){z_2}}
\end{array}} \right).\]
Here, ${D^2}f={\left( {\frac{{{\partial ^2}f}}{{\partial {x_i}\partial {x_j}}}} \right)_{1 \le i,j \le N}}$ is the Hessian matrix of a quadratic form
 at $V^*$.

\subsection{Bogdanov-Takens-Cusp bifurcation.} 
\label{sec_BTC}

The steady state $V^*$ becomes degenerate Bogdanov-Takens point 
(or ``Bogdanov-Takens-Cusp point"-BTC point)
 when a BT point combines with a Cusp.
A BTC occurs if either:  Case 1: $\alpha_2=0$ and $\beta_2\ne 0$; or Case 2: $\alpha_2\ne0$ and $\beta_2= 0$, see e.g. \cite{kuznetsov2005practical}. 
 Considering Case 1, and applying an approach similar to \cite{kirst2015fundamental} with the results of \cite{kuznetsov2005practical},  
 we have the following. 
\begin{theorem}\label{Th:BTCpoint}
Assume that $V^*$ is an ordinary BT  point. If
\begin{equation}\label{eq:dd_I_infter}
    {\left. {\frac{d^2}{{dV^2}}{I_\infty }(V)} \right|_{{V^*}}} = 0,
\end{equation}
then $\alpha_2=0$ and $\beta_2\ne0$, that is, $V^*$ becomes a Cusp.
\end{theorem}
\begin{proof}
From $f_m$ and $f_a$, we have
\[\frac{{\partial {f_2}}}{{\partial w}} = \frac{{ - C_M}}{r} \Rightarrow \frac{{{\partial ^2}{f_1}}}{{\partial {w^2}}} = 0\quad \text{and}\quad \partial _a^{{f_3}} =  - C_M{\tau ^{ - 1}} \Rightarrow \partial _{aa}^{{f_3}} = 0\]
Hence, the components of $G$ are 
\begin{neweq_non}
\frac{1}{{q_{01}^2}}q_0^T{D^2}{f_1}({V^*}){q_0} &= \frac{{{\partial ^2}{f_1}}}{{\partial {V^2}}} +  \frac{2r}{C_M}\frac{{{\partial ^2}{f_1}}}{{\partial V\partial w}}\frac{{\partial {f_2}}}{{\partial V}} + \frac{r^2}{C_M^2}\frac{{{\partial ^2}{f_1}}}{{\partial {w^2}}}{\left( {\frac{{\partial {f_2}}}{{\partial V}}} \right)^2}\\
&~~~+ \frac{2}{C_M}\partial _{Va}^{f_1^T}\tau \partial _V^{{f_3}} + \frac{1}{C_M^2}\partial _V^{f_3^T}\tau \partial _{aa}^{{f_1}}\tau \partial _V^{{f_3}}\\
\frac{1}{{{q^2_{01}}}}q_0^T{D^2}{f_2}({V^*}){q_0} &= \frac{{{\partial ^2}{f_2}}}{{\partial {V^2}}} + \frac{2C_M^2}{r^2}\frac{{dr}}{{dV}}\frac{{\partial {f_2}}}{{\partial V}}\\
\frac{1}{{{q^2_{01}}}}q_0^T{D^2}{f_3}({V^*}){q_0}&=\partial _{VV}^{{f_3}}  + 2C_M^2{\tau ^{ - 2}}\partial _V^\tau \tau \partial _V^{{f_3}}
\end{neweq_non}
where $\partial _{aa}^{{f_1}} = diag(\partial _{{a_3}{a_3}}^{{f_1}}, \ldots ,\partial _{{a_N}{a_N}}^{{f_1}})$. 
Consequently, 
\begin{neweq_non}
\frac{1}{{{p_{11}}q_{01}^2}}{\alpha _2} &= \frac{{{\partial ^2}{f_1}}}{{\partial {V^2}}} + \frac{2r}{C_M}\frac{{{\partial ^2}{f_1}}}{{\partial V\partial w}}\frac{{\partial {f_2}}}{{\partial V}} + \frac{r^2}{C_M^2}\frac{{{\partial ^2}{f_1}}}{{\partial {w^2}}}{\left( {\frac{{\partial {f_2}}}{{\partial V}}} \right)^2} + \frac{2}{C_M}\partial _{Va}^{f_1^T}\tau \partial _V^{{f_3}}\\
&+\frac{1}{C_M^2} \partial _V^{f_3^T}\tau \partial _{aa}^{{f_1}}\tau \partial _V^{{f_3}}+ \frac{r}{C_M}\frac{{\partial {f_1}}}{{\partial w}}\frac{{{\partial ^2}{f_2}}}{{\partial {V^2}}} + \frac{2C_M}{r}\frac{{\partial {f_1}}}{{\partial w}}\frac{{dr}}{{dV}}\frac{{\partial {f_2}}}{{\partial V}}\\
&~  +\frac{1}{C_M} \partial _a^{f_1^T}\tau \partial _{uu}^{{f_3}} + 2C_M\partial _a^{f_1^T}{\tau ^{ - 1}}\partial _V^\tau \tau \partial _V^{{f_3}}.
\end{neweq_non}
Recall that all of these derivatives are calculated at $V^*$.  It follows from (\ref{Eq:deriv}) that
\begin{neweq_non}
{\left. {\frac{{{\partial ^2}{f_2}}}{{\partial {V^2}}}} \right|_{{E^*}}} &= C_M{r^{ - 1}}(V^*)\left( {{{\left. {\frac{{{d^2}}}{{d{V^2}}}{w_\infty }(V)} \right|}_{{V^*}}}} \right)\\
&~~~~-2 C_M^2{r^{ - 2}}\left( {{{\left. {\frac{d}{{dV}}r(V)} \right|}_{{V^*}}}} \right)\left( {{{\left. {\frac{d}{{dV}}{w_\infty }(V)} \right|}_{{V^*}}}} \right),\\[5pt]
{\left. {\partial _{VV}^{{f_3}}} \right|_{{E^*}}} &= C_M{\tau ^{ - 1}}(V^*){\left. {\partial _{VV}^{{a_\infty }}} \right|_{{V^*}}} -2C_M^2 {\tau ^{ - 2}}(V^*)\partial _V^\tau {\left. {\partial _V^{{a_\infty }}} \right|_{{V^*}}}.
\end{neweq_non}
At $V=V^*$, we have $\frac{{\partial {f_2}}}{{\partial V}} =C_M {r^{ - 1}}\frac{{d{w_\infty }}}{{dV}}$ and $\partial _V^{{f_3}} = C_M{\tau ^{ - 1}}\partial _V^{{a_{\infty}}}$. Hence, 
\begin{neweq_non}
\frac{1}{{{p_{11}}q_{01}^2}}{\alpha _2} &= \frac{{{\partial ^2}{f_1}}}{{\partial {V^2}}} + 2\frac{{{\partial ^2}{f_1}}}{{\partial V\partial w}}\frac{{d{w_\infty }}}{{dV}} + \frac{{{\partial ^2}{f_1}}}{{\partial {w^2}}}{\left( {\frac{{d{w_\infty }}}{{dV}}} \right)^2}\\
&+ \frac{{\partial {f_1}}}{{\partial w}}\frac{{{d^2}{w_\infty }}}{{d{V^2}}} + 2\partial _{Va}^{f_1^T}\partial _V^{{a_\infty }} + \partial _V^{a_\infty ^T}\partial _{aa}^{{f_1}}\partial _V^{{a_\infty }} + \partial _a^{f_1^T}\partial _{VV}^{{a_\infty }}\\
&={\left. {\frac{d^2}{{dV^2}}{I_\infty }(V)} \right|_{{V^*}}}.
\end{neweq_non}
Thus, $\alpha_2=0$ if and only if 
\[
{\left. {\frac{d^2}{{dV^2}}{I_\infty }(V)} \right|_{{V^*}}}=0.
\]
Consequently, from (\ref{eq:h20}), we have $A{h_{20}} = -G(q_0,q_0)$, which has an infinite solutions. This system is consistent due to the Fredholm solvability condition \cite{kuznetsov2005practical}. 
Hence,  $h_{20}$ can be chosen such that $\beta_2\ne0$  in (\ref{eq:alphabeta}). 
This completes the proof.
\end{proof}

\section{Existence of the bifurcations} \label{sec:exist}

Theorems \ref{Th:BTpoint} and \ref{Th:BTCpoint} imply three bifurcations: BT, CP and BTC 
which are characterized by equations (\ref{eq:I_star}-\ref{BT:Con_A2}) and (\ref{eq:dd_I_infter}). 
In the following we discuss the solution of these equations.  Recall that equation \eqref{eq:I_star} relates the
equilibrium point voltage value $V^*$ to $I_{app}$ and the other parameters.

Rearranging (\ref{BT:Con_A1}), we obtain
\begin{equation}\label{eq:param1}
-g_M X_1(V^*) + X_2(V^*)=g_L
\end{equation}
where
\begin{align*}
  X_1(V^*)&={w^*} + \left( {{{\left. {\frac{d}{{dV}}{w_\infty }(V)} \right|}_{{V^*}}}} \right)\left( {{V^*} - V_K} \right),\\
  X_2(V^*)&={\left. {\frac{d}{{dV}}{I_{ion,\infty }}(V)} \right|_{{V^*}}}.  
\end{align*}
Similarly,   (\ref{BT:Con_A2}) leads to 
\begin{equation}\label{eq:param2}
-g_M Y_1(V^*) +Y_2(V^*)=-1
\end{equation}
where
\begin{align*}
    Y_1(V^*)&=\frac{r(V^*)}{C_M}\left( {{V^*} - V_K} \right)\left( {{{\left. {\frac{d}{{dV}}{w_\infty }(V)} \right|}_{{V^*}}}} \right),\\
    Y_2(V^*)&=\frac{1}{C_M}{\left. {\partial _a^{{I_{ion}}}} \right|_{{E^*}}}\tau (V^*){\left. {\partial _V^{{a_\infty }}} \right|_{{V^*}}}.
\end{align*}
It is easy to check that the second derivative of $I_{\infty }(V)$ is 
\begin{align*}
{\left. {\frac{{{d^2}}}{{d{V^2}}}{I_\infty }(V)} \right|_{{V^*}}} &=  - {g_M}\left[ {2{{\left. {\frac{{d{w_\infty }}}{{dV}}} \right|}_{{V^*}}} + {{\left. {\frac{{{d^2}{w_\infty }}}{{d{V^2}}}} \right|}_{{V^*}}}({V^*} - V_K)} \right] \\
&~\qquad+ {\left. {\frac{{{d^2}}}{{d{V^2}}}{I_{ion,\infty }}(V)} \right|_{{V^*}}}
\end{align*}
Thus, (\ref{eq:dd_I_infter}) holds when 
\begin{equation} \label{eq:param3}
-g_M Z_1(V^*)+Z_2(V^*)=0.
\end{equation}

{\bf Bogdanov-Takens Bifurcation}.  Suppose there is $V^*$ that satisfies
\[ [g_L-X_2(V^*)]Y_1(V^*)+X_1(V^*)(Y_2(V^*)+1)=0,\]
with at least one of $X_1(V^*),Y_1(V^*)$ nonzero.
Then there is an equilibrium $E^*=(V^*,w^*,{\bf a}^*)$ that undergoes a 
Bogdanov-Takens bifurcation at
$(I_{app},g_M)=(I_{app}^*,g_M^*)$ where  
\begin{equation}
g_M^*= \frac{X_2(V^*)-g_L}{X_1(V^*)}= \frac{Y_2(V^*)+1}{Y_1(V^*)}
\label{gMstar}
\end{equation}
and $I_{app}^*$ is given by \eqref{eq:I_star}.

{\bf Cusp Bifurcation}. Suppose there is $V^*$ that satisfies
\[ [g_L-X_2(V^*)]Z_1(V^*)+X_1(V^*)Z_2(V^*)=0\]
with at least one of $X_1(V^*),Z_1(V^*)$ nonzero.
Then there is an equilibrium $E^*=(V^*,w^*,{\bf a}^*)$ that undergoes a Cusp bifurcation at
$(I_{app},g_M)=(I_{app}^*,g_M^*)$  where
\begin{equation}
g_M^*=\frac{X_2(V^*)-g_L}{X_1(V^*)}= \frac{Z_2(V^*)}{Z_1(V^*)}
\label{gMCusp}
\end{equation}
and $I_{app}^*$ is given by \eqref{eq:I_star}.

{\bf Bogdanov-Takens-Cusp Bifurcation}. Suppose there is $V^*$ that satisfies
\[ [Y_2(V^*)+1]Z_1(V^*)-Y_1(V^*)Z_2(V^*)=0\]
with at least one of $Y_1(V^*),Z_1(V^*)$ nonzero.
Then there is an equilibrium $E^*=(V^*,w^*,{\bf a}^*)$ that undergoes a BTC bifurcation at
$(I_{app},g_M,g_L)=(I_{app}^*,g_M^*,g_L^*)$ 
where 
\begin{eqnarray*} 
g_M^*&=&\frac{Y_2(V^*)+1}{Y_1(V^*)}=\frac{Z_2(V^*)}{Z_1(V^*)}\\
g_L^*&=&X_2(V^*)-g_M^*X_1(V^*)
\end{eqnarray*}
and $I_{app}^*$ is given by \eqref{eq:I_star}.

\begin{remark}
We have explicitly included the leak current in our formulation. 
The leak current is not necessary for the occurrence of the BT and CP 
bifurcations. 
If $g_L=0$ then equations \eqref{eq:param1} becomes
\[ -g_M X_1(V^*) + X_2(V^*)=0 \]
and the solution will go through as above.  However for the BTC bifurcation 
to occur we must have another parameter to vary. We have shown that this 
third parameter can be the leak conductance, $g_L$.  Solving the equations
in a different way shows that the capacitance, $C_M$ could also be used.
\end{remark}
     
\subsection{Implications.}

In the previous section we gave conditions which guarantee  that a BT bifurcation 
can be induced by the variation of two parameters found in our general model: the applied current, $I_{app}$, and the maximal conductance of the M-current $g_M$. Further, if the conditions are met, we gave explicit expressions for the bifurcation point in terms of $g_M$ and $I_{app}$.
Near this bifurcation point the behaviour of the system will be described by the unfolding of the normal form \eqref{normal_form_1} in terms of two parameters.  The normal form and unfolding were first studied by 
\cite{bognadov1975versal,takens1974singularities}. The details can be found also be found in  \cite{Guck_Holmes,kuznetsov2013elements}. A key point for our work is that emanating out of the BT point are three codimension-one bifurcation curves: Hopf bifurcation, saddle homoclinic bifurcation and saddle node (fold) of equilibria. A periodic orbit exists between
the Hopf and homoclinic bifurcation curves, the stability of which depends on the sign of the coefficients $\alpha_2,\beta_2$ in \eqref{normal_form_1}. Thus the emergence of periodic solutions via a Hopf bifurcation can be linked to the presence of the BT point.

In the previous section, we also gave conditions which guarantee that a BTC bifurcation can be induced by $I_{app}$, $g_M$ and the conductance of the leak current, $g_L$. 
The normal form and unfolding for the case considered in Theorem~\ref{Th:BTCpoint} was first studied in \cite{dumortier2006bifurcations}; see also \cite{kuznetsov2005practical,kirst2015fundamental}. There are various possibilities for the bifurcations in the unfolding which are determined by the higher order terms in the normal form. The key results for our analysis are that in the three dimensional parameter space there are two curves of cusp bifurcations and two curves of BT bifurcations with a surface of Hopf bifurcation starting at one BT curve and ending at the other. Near one BT bifurcation the Hopf bifurcation is supercritical (produces an asymptotically stable periodic orbit), while at the other it is subcritical. There is a saddle-node (fold) of limit cycles bifurcation associated with the change in criticality of the Hopf bifurcation. Fixing the value of one parameter (such as the leak conductance, $g_L$,) amounts to taking a two dimensional slice in the three dimensional parameter space. Thus, in general one should expect to see some subset of the bifurcations we have just described.

\section{Numerical Examples}
\label{sec_Numerical}

In this section, we implement three examples with different ranges of $g_M$ corresponding to the range between the BT  and Cusp points. 
We apply our theoretical results and compare them with computations that carried out with \textsf{MATCONT}   \cite{dhooge2003matcont}.
We also construct bifurcation diagrams using \textsf{MATCONT} to explain  the possible behaviour of each example. The labels used these bifurcation diagrams  are summarized in Table~\ref{Table:labels}.
Furthermore, we use the numerical solution of the model in each example to measure the frequency-current (${\rm{ F/I}}$) curves which illustrate the neuronal excitability class.

\begin{table}
    \centering
    \begin{tabular}{|c|c|}
    \hline
    Label & Bifurcation\\
    \hline
    LP   & limit point (fold/saddle-node) of equilibria \\
      red/black H  & super/subcritical Andronov-Hopf  \\
      LPC &  limit point (fold) of cycles\\
    \hline
    BT & Bogdanov-Takens \\ 
    CP & Cusp \\ 
    GH & generalized Hopf (Bautin) \\
    \hline
    \end{tabular}
    \caption{Labels used to mark bifurcation points in one and two parameter bifurcation diagrams.}
    \label{Table:labels}
\end{table}

\pdfbookmark[subsection]{Example 1: Wang--Buzsaki model}{Example 1: Wang--Buzsaki model}
\subsection*{Example 1.} In \cite{wang1996gamma}, Wang and Buz\'saki proposed a model to study the fast neuronal oscillations in the neocortex and hippocampus during behavioral arousal.  The model is based on an inhibitory basket cell in rat hippocampus.
The model with the inclusion of the $M-$current can be written as 
\begin{align}
	\label{WB_model}
	C_m\frac{{dV}}{{dt}} &= {I_{app}} - {g_L}(V - {V_L}) - {g_M}w(V - {V_K}) - {g_{Na}}m_\infty ^3\left( V \right)h(V - {V_{Na}})\nonumber\\
	&~ - {g_K}{n^4}(V - {V_K}),\nonumber\\
	\frac{{dw}}{{dt}} &= \frac{1}{{{\tau _w}(V)}}\left( {{w_\infty }(V) - w} \right),\\
	\frac{{d\sigma}}{{dt}} &= \frac{\phi}{{{\tau _{\sigma}}(V)}}\left( {{{\sigma}}_\infty }(V) - \sigma \right),\quad \sigma\in\{h,n\},\nonumber
\end{align}
supplemented by the dynamics for the gating variables $h$ and $n$ as in (\ref{Model}). 
Parameter values and other details of the model are given in the Appendix.

\sloppy Figure~\ref{fig:WB_BT_CP} shows the contour plot of  equations~\eqref{BT:Con_A1}, \eqref{BT:Con_A2} and \eqref{eq:dd_I_infter}. In Figure~\ref{fig:WB_BT}, there are two intersections 
of equations \eqref{BT:Con_A1} and \eqref{BT:Con_A2} at
$g_M=-0.0368$ and $g_M=0.1455$. Consequently, there are two BT points: $(V^*,I_{app}^*,g_M^*)=(-40.9926,-6.7925,-0.0368)$ and $(-59.6978,0.2000,0.1455)$. The bio-physically permissible point is the latter  one where $g_M>0$. 
Moreover,  there is one intersection of (\ref{BT:Con_A2}) and (\ref{eq:dd_I_infter}) implying the  Cusp point is $(\widehat{V},\widehat{I}_{app},\widehat{g}_M)=(-51.5531,1.2382,2.3316)$, see Figure \ref{fig:WB_BT}. 

The analysis of section~\ref{sec:exist} shows that of the three curves, only the one defined by eq.~\ref{BT:Con_A1} depends on $g_L$. Further, the representation \eqref{eq:param1} of this equation and the properties of the M-current show that increasing $g_L$ will move this curve downward. Given the shape of the curves in
Figure~\ref{fig:WB_BT}, it is clear that increasing $g_L$ will move the BT and CP points closer together, and for sufficiently large $g_L$ we should obtain a single intersection point of all three curves, corresponding to a BTC point.
Figure \ref{fig:WB_BTC_A} confirms that when we increase $g_L$ to $0.7507$, we find the (approximate) BTC  point  $(-46.6416,7.75907,-0.0166046)$.

\begin{figure}[htb!]
\centering
\begin{subfigure}{.45\textwidth}
  \centering
  \includegraphics[width=0.95\linewidth]{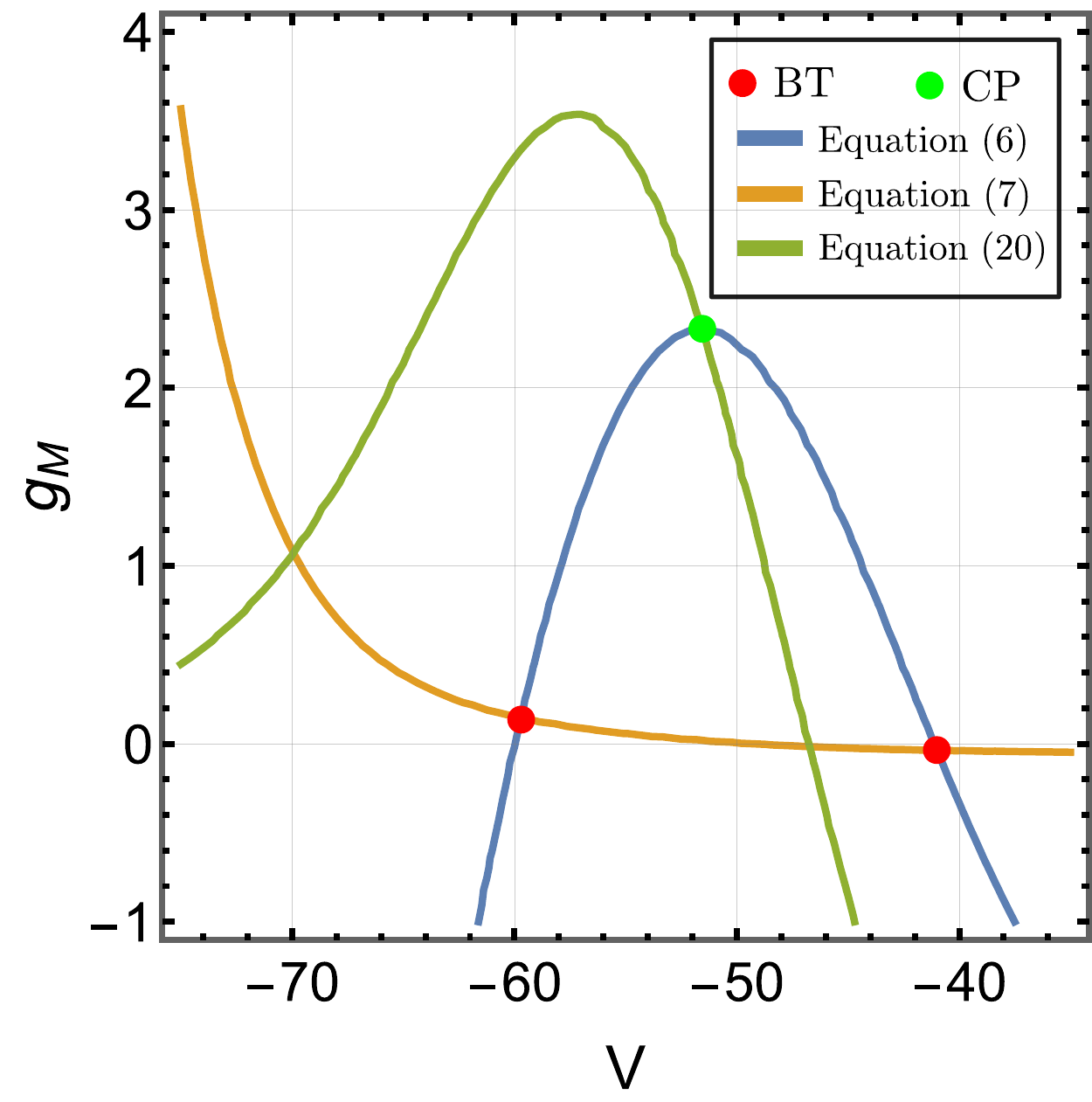}
 \caption{BT and CP points}
 \label{fig:WB_BT}
\end{subfigure}
\begin{subfigure}{.45\textwidth}
  \centering
  \includegraphics[width=0.94\linewidth]{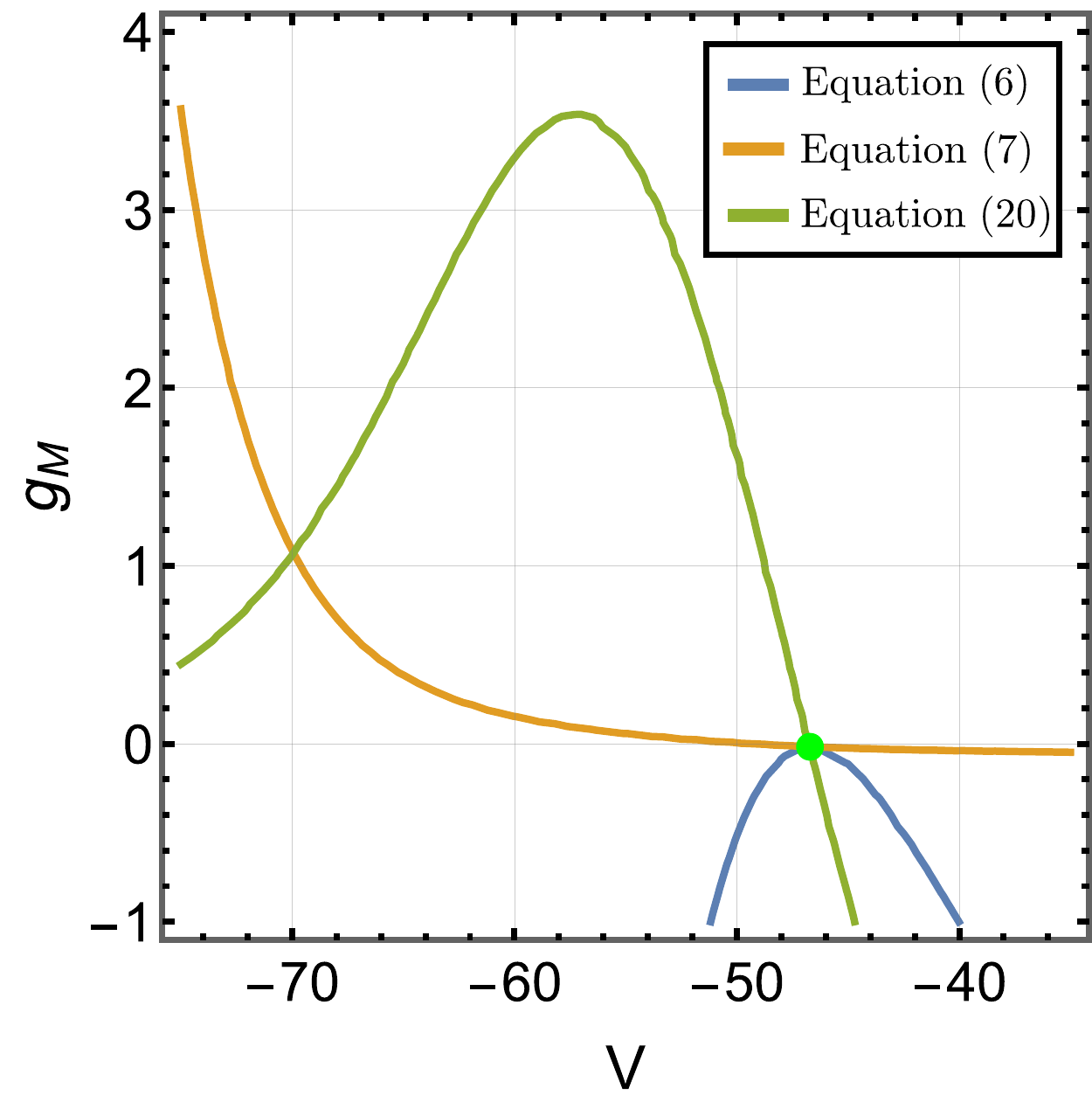}
 \caption{BTC point}
   \label{fig:WB_BTC_A}
\end{subfigure}
\caption{Existence of codimension two and three bifurcation points in the Wang-Buz\'saki model (\ref{WB_model}), with the parameter values given in Table \ref{WB_parameters}. (a) The conditions given by equations~ \eqref{BT:Con_A1}, \eqref{BT:Con_A2} and \eqref{eq:dd_I_infter} are plotted in the $V,g_M$ space. The two intersection points (red dots) of the conditions in Theorem~\ref{Th:BTpoint} show that there are two BT points in the model. 
The one intersection point (green dot) of the conditions in Theorem~\ref{Th:BTCpoint} show the existence of one Cusp point; (b) The three conditions are plotted when the leak conductance is increased to $g_L=0.7507$. The intersection point (green dot) corresponds to the BTC point.}
\label{fig:WB_BT_CP}
\end{figure}

\begin{figure}[htb!]
\centering
\begin{subfigure}{.40\textwidth}
  \centering
  \includegraphics[width=0.95\linewidth]{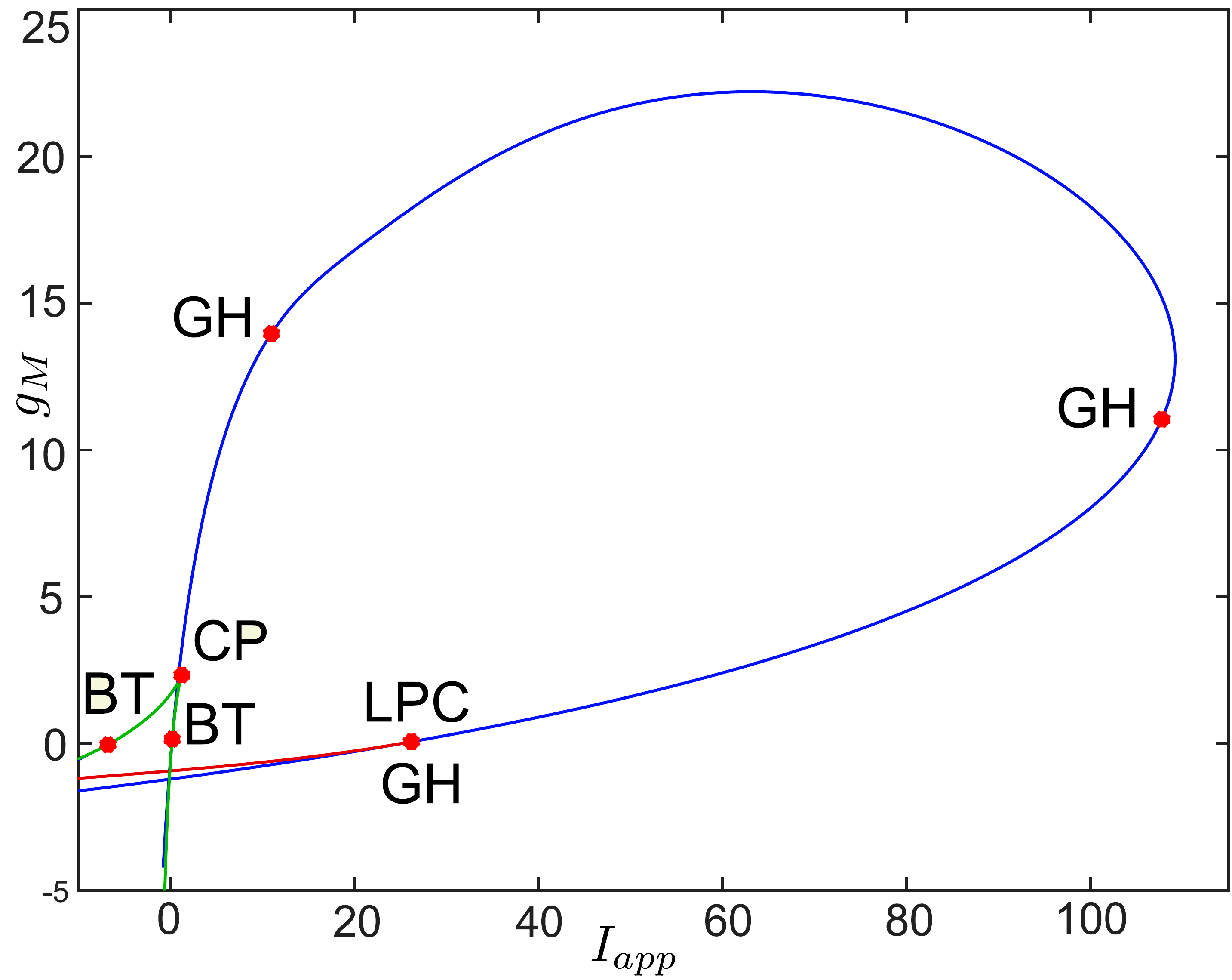}
  \caption{}
  \label{fig:WB_model_gMIapp_a}
\end{subfigure}
\begin{subfigure}{.40\textwidth}
  \centering
  \includegraphics[width=0.95\linewidth]{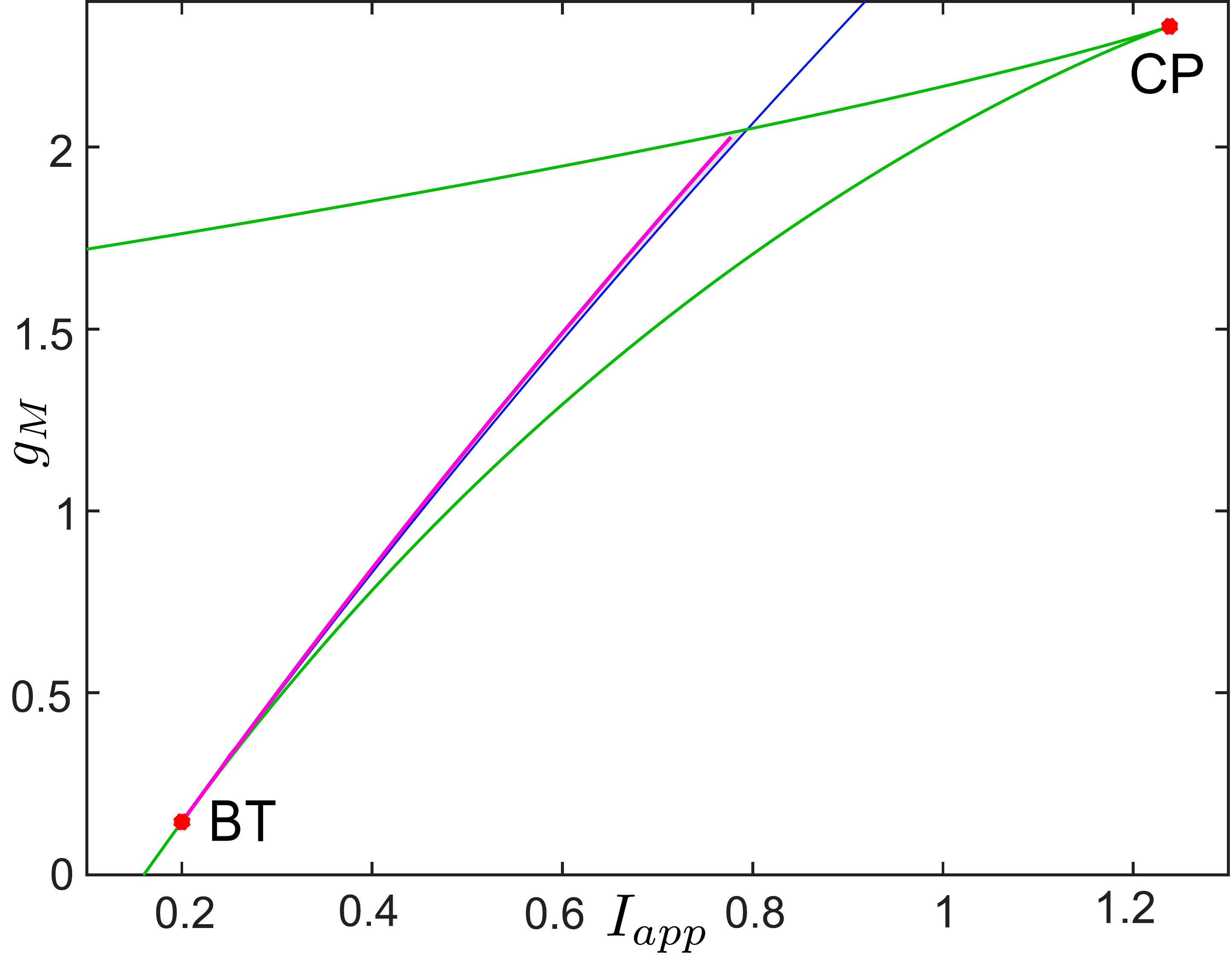}
   \caption{}
 \label{fig:WB_model_gMIapp_b}
\end{subfigure}
\caption{Bifurcation diagram in $I_{app},g_M$ parameter space for the Wang--Buz\'saki model (\ref{WB_model}). Green curves are limit point (fold/saddle-node) bifurcations of equilibria, blue are Andronov-Hopf bifurcations, magenta are homoclinic bifurcations and red  are limit point (fold) bifurcations of limit cycles. Co-dimension two bifurcation point labels are described in Table~\ref{Table:labels}.
}
\label{fig:WB_model_gMIapp}
\end{figure}

We use  the MATLAB numerical continuation package \textsf{MATCONT}   \cite{dhooge2003matcont} to verify the theoretical results and supplement them by  numerical bifurcation diagrams. 
From  \textsf{MATCONT}, we find two BT  points $(V^*,I_{app}^*,g_M^*)=(-59.698,0.2,0.146)$ and $(-40.992,-6.792,-0.036)$ (we omit this point) and one  Cusp point $(\widehat{V},\widehat{I}_{app},\widehat{g}_M)=(-51.553,1.238,2.332)$ with parameter values in Table \ref{WB_parameters}. 
This is consistent with our results in Figure \ref{fig:WB_BT_CP}.

Now, we discuss  the switch in the model neuronal excitability class as $g_M$ increases. 
We plot a bifurcation diagram in $I_{app},g_M$ parameter space for Wang--Buz\'saki model (\ref{WB_model}) in Figure \ref{fig:WB_model_gMIapp}.  As expected from normal form analysis, there is a curve of homoclinic bifurcations, a curve of Hopf bifurcation and a curve of saddle-node of equilibria emanating from the BT point. The Hopf is subcritical and thus an unstable periodic orbit exists for any parameters between the homoclinic and Hopf curves.
See Figure \ref{fig:WB_model_gMIapp_b}.
These curves are associated with the transition in the neuronal excitability class
and show three cases.  
\begin{itemize}
    \item $g_M<g_M^*$: In Figure \ref{fig:WB_model_V_Iapp_a}, when  $g_M=0$ and $I_{app}<0.16$, there exists a stable equilibrium point, that determines the resting state, and two unstable equilibria. 
As the applied current increases, the stable and one unstable
fixed points collide in a saddle-node bifurcation point (``LP").
Consequently,  a limit cycle is born simultaneously and emanates from the LP , that is, the limit cycle is created via a
saddle-node on invariant circle bifurcation  (``SNIC"). 
As expected, the oscillations on the limit cycle appear with arbitrarily slow frequency (see Figure \ref{fig:WB_model_FI_Im1}a),  indicating  Class-I excitability \cite{izhikevich2007dynamical,rinzel1998analysis}.

  \item  $g_M>\widehat{g}_M$: For large enough $g_M$, a different sequence of bifurcations is observed.
  In Figure \ref{fig:WB_model_V_Iapp_c}, when  $g_M=3$, 
  at $I_{app}=1.$, a limit point bifurcation of cycles, ``LPC", occurs giving rise to one unstable and one stable periodic orbit. Then, at $I_{app}=1.1416$, the unstable periodic orbit disappears in a subcritical Hopf bifurcation (subHopf) of the lone equilibrium point, destabilizing it.
  Consequently, firing with a positive frequency appears via LPC, and hence, neuronal excitability  Class-II   occurs \cite{izhikevich2007dynamical,rinzel1998analysis}. See Figures  \ref{fig:WB_model_V_Iapp_c}-\ref{fig:WB_model_FI_Im1}c; 

 \item  $g_M^*<g_M<\widehat{g}_M$: In this case, both subHopf and LP  exist. The stable equilibrium point disappears by subHopf and  the LP   occurs when two unstable equilibria collide.
The model dynamics exhibit two different patterns, which are only distinguished by the bifurcations of the unstable periodic orbit(s). (i) When $g_M^*<g_M<2.1$, 
see Figures \ref{fig:WB_model_FI_Im2}a-\ref{fig:WB_model_FI_Im2}c and in Figure \ref{fig:WB_model_V_Iapp_b},
an unstable limit cycle is created via a homoclinic bifurcation (the magenta curve in Figure~\ref{fig:WB_model_gMIapp_b}) and disappears in the subHopf. 
In this case, the stable limit cycle appears via an LPC with a different unstable limit cycle which disappears via homoclinic orbit bifurcation (not shown in  Figure~\ref{fig:WB_model_gMIapp_b}). 
(ii) When $2.1\lesssim g_M<\widehat{g}_M$, the sequence
of bifurcations is very similar to that for $g_M>\widehat{g}_M$. An LPC bifurcation creates and
unstable and stable periodic orbit. The former is lost
in the subHopf, see  Figure \ref{fig:WB_model_FI_Im2}c.
For all $g_M\in (g_M^*,\widehat{g}_M)$ there is a region of bistability between a stable limit cycle and a stable equilibrium point, between the LPC and subHopf bifurcations. 
Consequently, when  $g_M^*<g_M<\widehat{g}_M$, a neuronal excitability  Class-II      occurs \cite{izhikevich2007dynamical,rinzel1998analysis}.  
\end{itemize}
Therefore, 
the model neuronal excitability type switches from  Class-I  to Class-II when the conductance of the M-current $g_M$ passes through the BT point.

Now let us consider the effect of the leak conductance, $g_L$. As shown above, increasing $g_L$ monotonically decreases the $g_M$ value  at the bio-physically permissible BT point. This means that the range of values of $g_M$ where the model has class I excitability will be decreased. Equivalently, smaller changes of $g_M$ are needed to switch the model from class I to class II.  If $g_L$ is increased enough then $g_M^*$ may become negative, in which case the model will exhibit class II excitability regardless of the value of $g_M$.

 \begin{figure}[htb!]
\centering
\begin{subfigure}{.315\textwidth}
  \centering
  \includegraphics[width=0.95\linewidth]{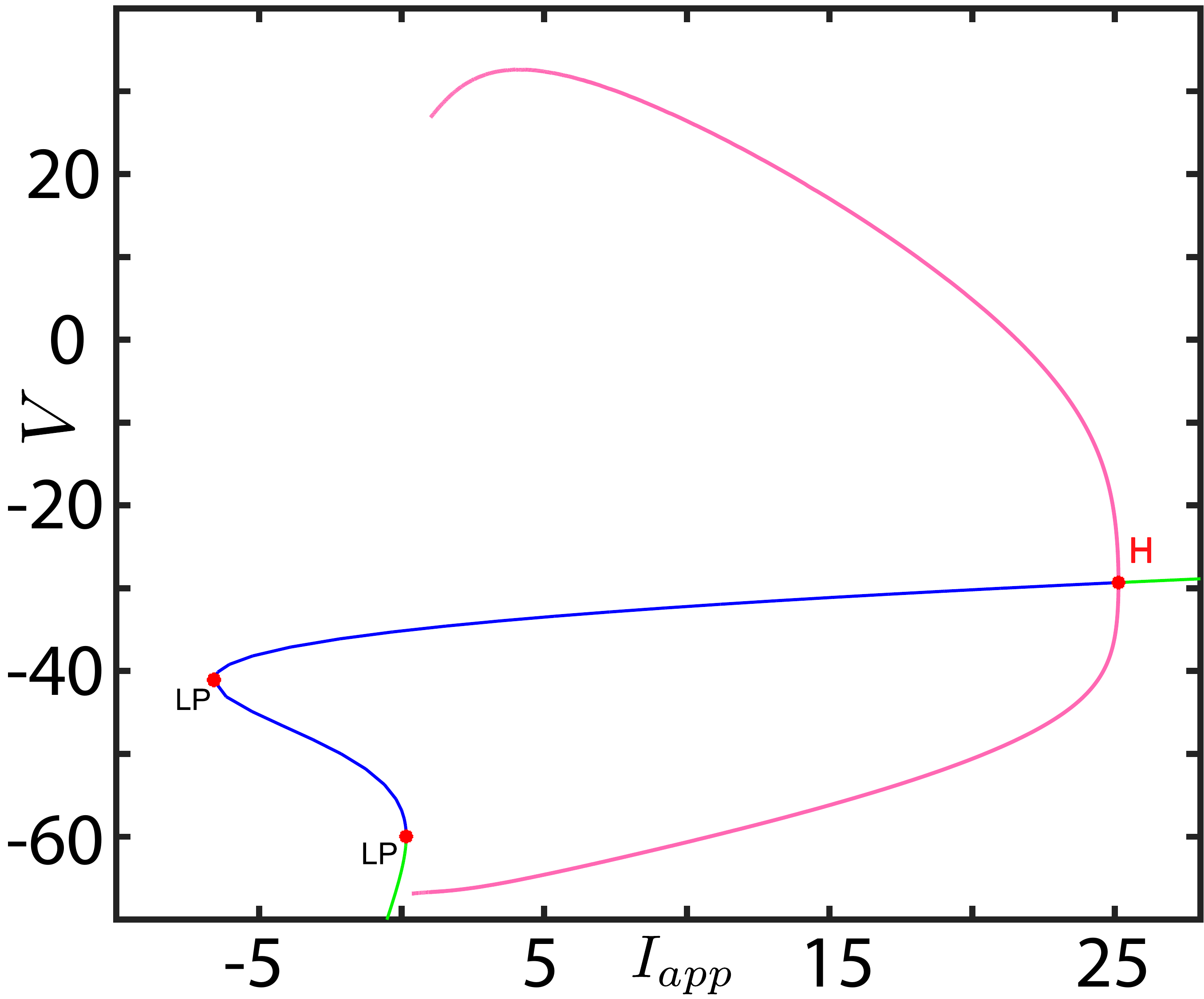}
    \caption{$g_M=0$}
    \label{fig:WB_model_V_Iapp_a}
\end{subfigure}
\begin{subfigure}{.315\textwidth}
  \centering
  \includegraphics[width=0.95\linewidth]{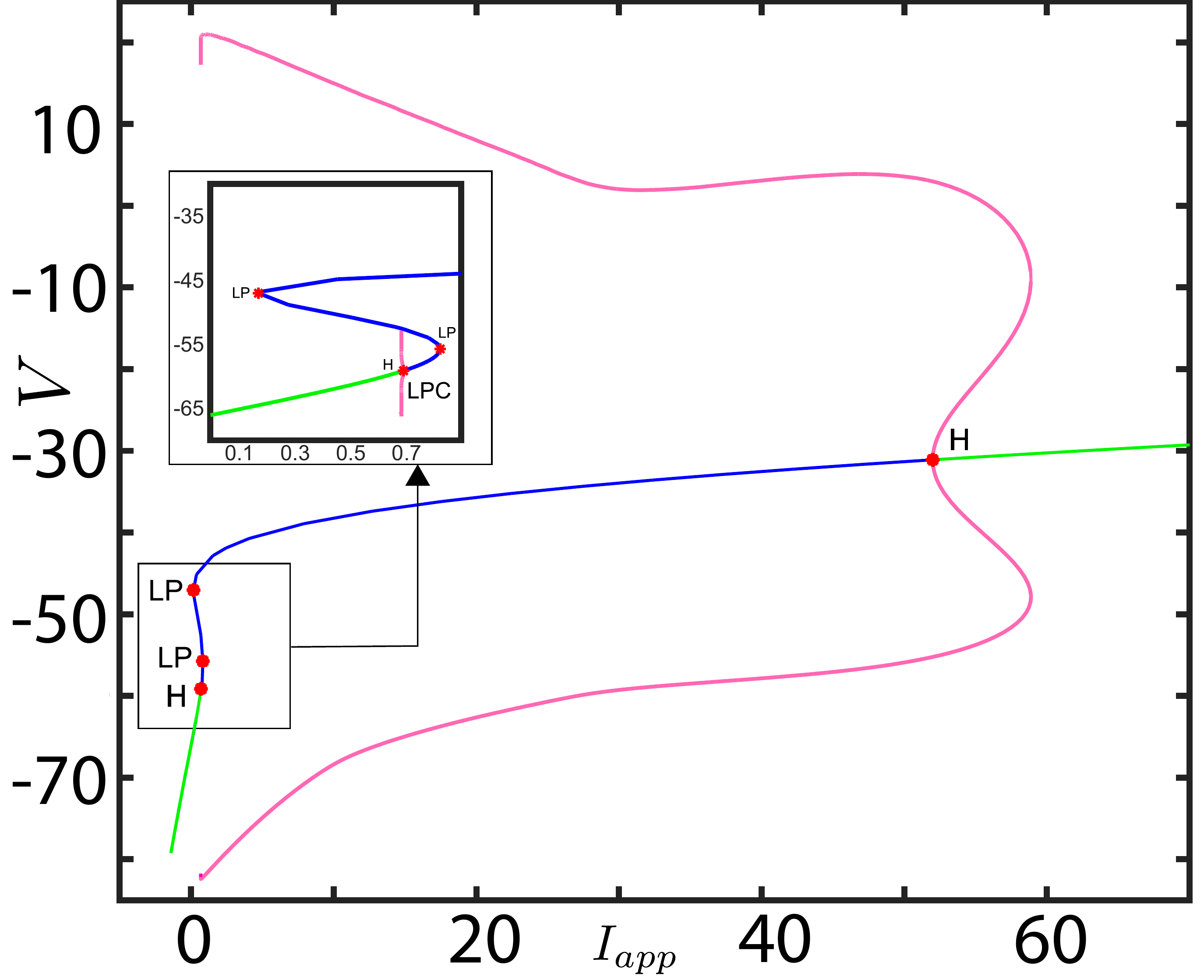}
   \caption{$g_M=1.75$}
       \label{fig:WB_model_V_Iapp_b}
\end{subfigure}
\begin{subfigure}{.315\textwidth}
  \centering
  \includegraphics[width=0.95\linewidth]{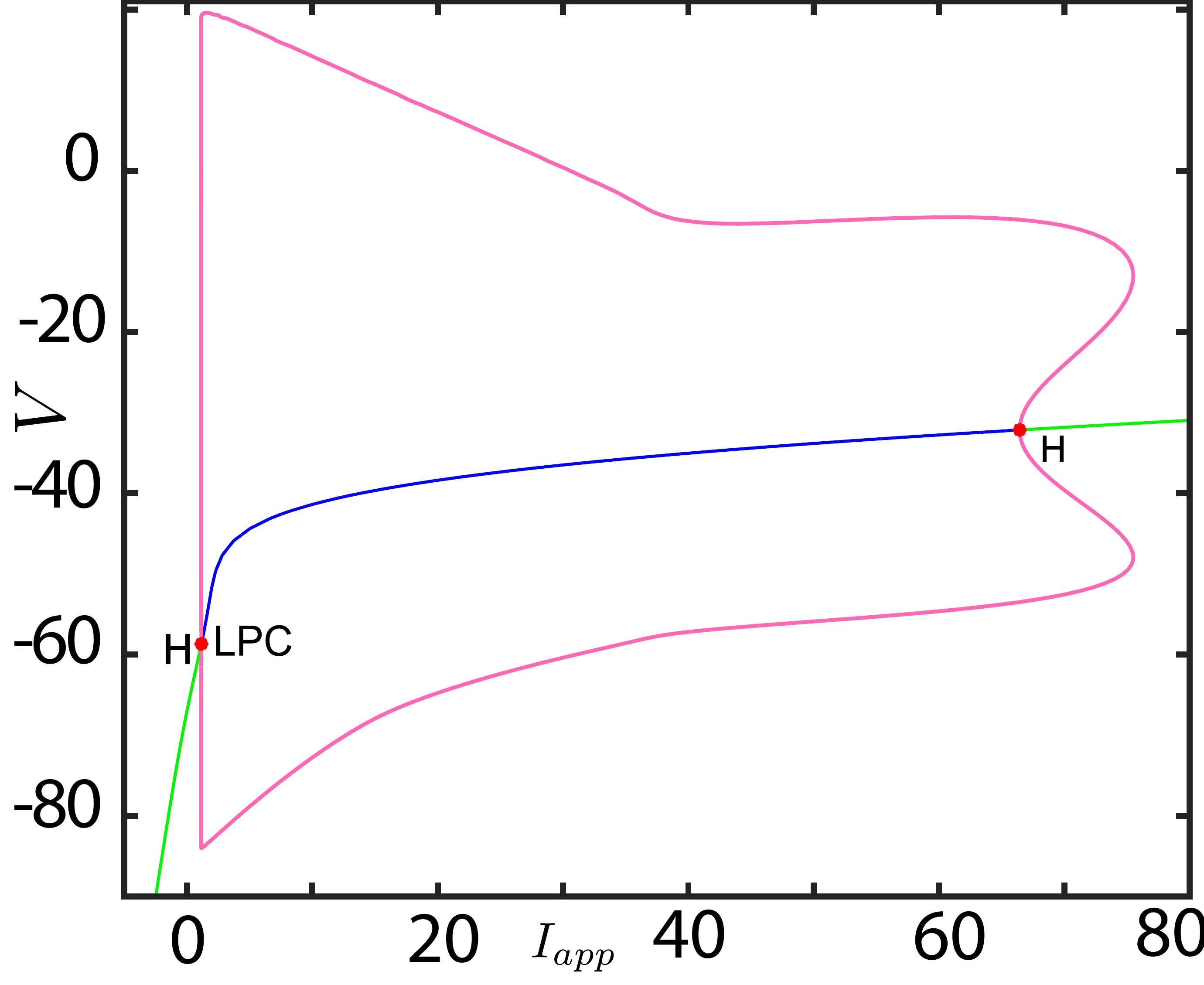}
   \caption{$g_M=3$}
       \label{fig:WB_model_V_Iapp_c}
\end{subfigure}

\caption{One parameter bifurcation diagrams for Wang--Buz\'saki model (\ref{WB_model}), showing the change in bifurcation structure as $g_M$ is varied. (a) $g_M<g_M^*$ (the value at the BT point); (b) $g_M^*<g_M<\widehat{g}_M$ 
(c) $g_M>\widehat{g}_M$ (the value at the CP point). Green/blue curves show stable/unstable equilibria. Pink curves show maxima/minima of periodic orbits. Co-dimension one bifurcation point labels are described in Table~\ref{Table:labels}.
}
\label{fig:WB_model_V_Iapp}
\end{figure}

\begin{figure}[htb!]
\centering
  \includegraphics[width=0.95\linewidth]{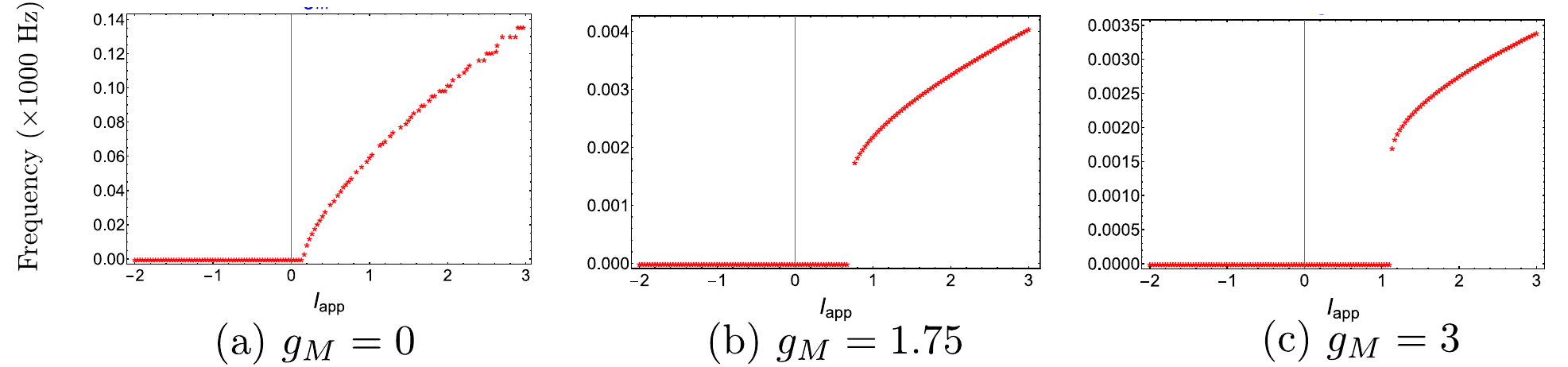}
    \caption{${\rm{ F/I}}$ curves of Wang--Buz\'saki model (\ref{WB_model}) corresponding to Figure~\ref{fig:WB_model_V_Iapp}.
    (a) $g_M<g_M^*$ (the value at the BT point); (b) $g_M^*<g_M<\widehat{g}_M$;
(c) $g_M>\widehat{g}_M$ (the value at the CP point).}
    \label{fig:WB_model_FI_Im1}
\end{figure}

\begin{figure}[htb!]
\centering
  \includegraphics[width=0.95\linewidth]{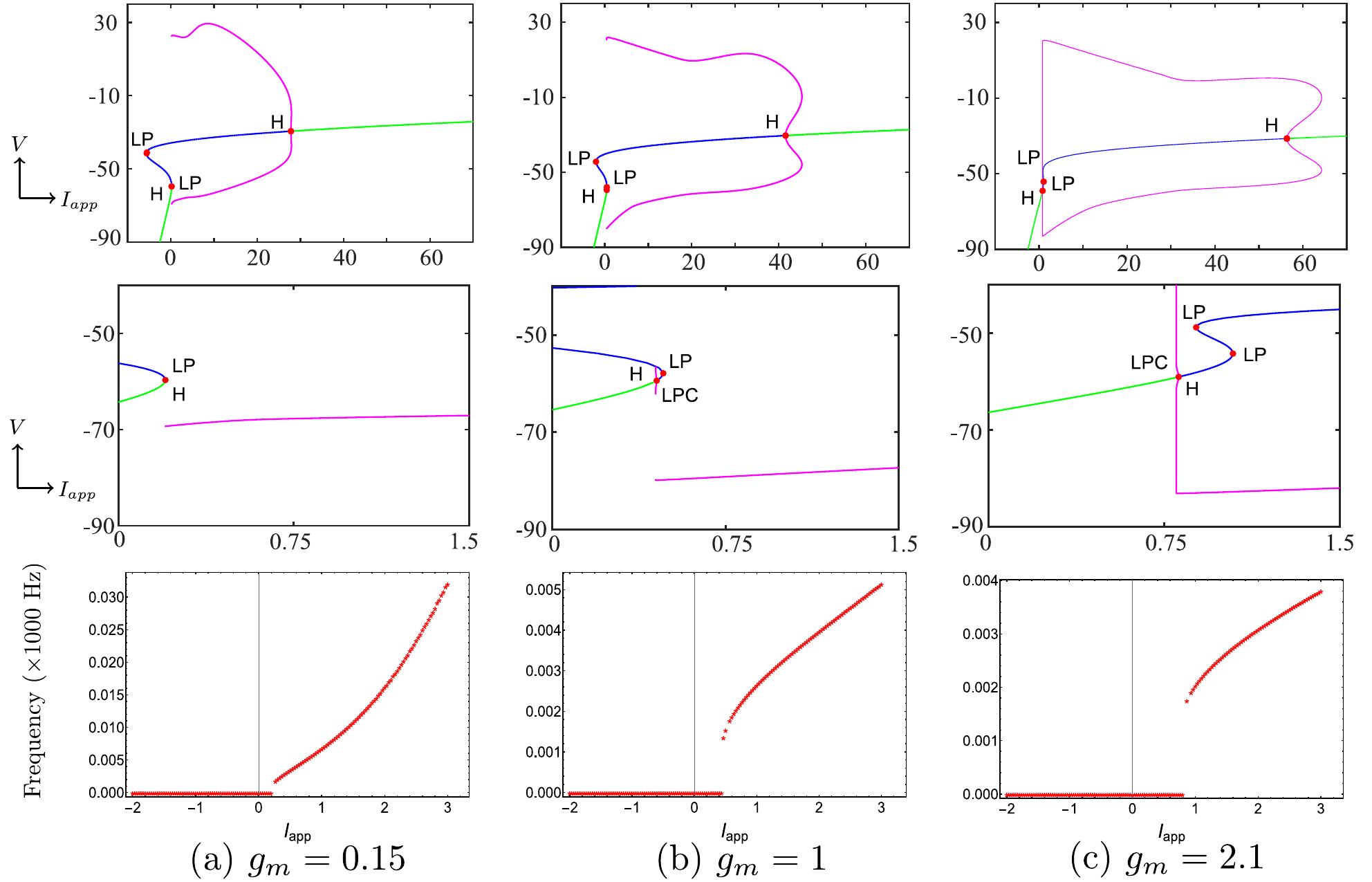}
    \caption{Top row/middle row: Details of the change in the bifurcation structure of the Wang--Buz\'saki model (\ref{WB_model}) when $g_M$ is varied between the BT point and the cusp point. Green/blue curves show stable/unstable equilibria. Pink curves show maxima/minima of periodic orbits. Co-dimension one bifurcation point labels are described in Table~\ref{Table:labels}. Bottom row: corresponding F/I curves.}
    \label{fig:WB_model_FI_Im2}
\end{figure}

\pdfbookmark[subsection]{Example 2: Stiefel Model}{Example 2: Stiefel Model}
\subsection*{Example 2.} In     \cite{stiefel2009effects}, Stiefel et al. proposed 
a single-compartmental neuron model that included biophysically realistic mechanisms for
neuronal spiking based on Hodgkin and Huxley ionic currents. 
The single-compartment  Stiefel Model can be written as: 
\begin{align}
	\label{Stiefel_model}
	C_m\frac{{dV}}{{dt}}& = {I_{app}} - {g_L}(V - {V_L}) - {g_M}w(V - {V_K}) - {g_{Na}}m_\infty ^3\left( V \right)h(V - {V_{Na}})\nonumber\\
	&~ - {g_{K}}{n^4}(V - {V_K}),\\
	\frac{{d\sigma}}{{dt}} &= \frac{\phi_{\sigma} }{{{\tau _{\sigma}}(V)}}\left( {{{\sigma}}_\infty }(V) - \sigma \right),\quad \sigma\in\{w,h,n\},\nonumber
\end{align}
Parameter values and other details can be found in the Appendix.

\begin{figure}[htb!]
\centering
\begin{subfigure}{.45\textwidth}
  \centering
  \includegraphics[width=0.95\linewidth]{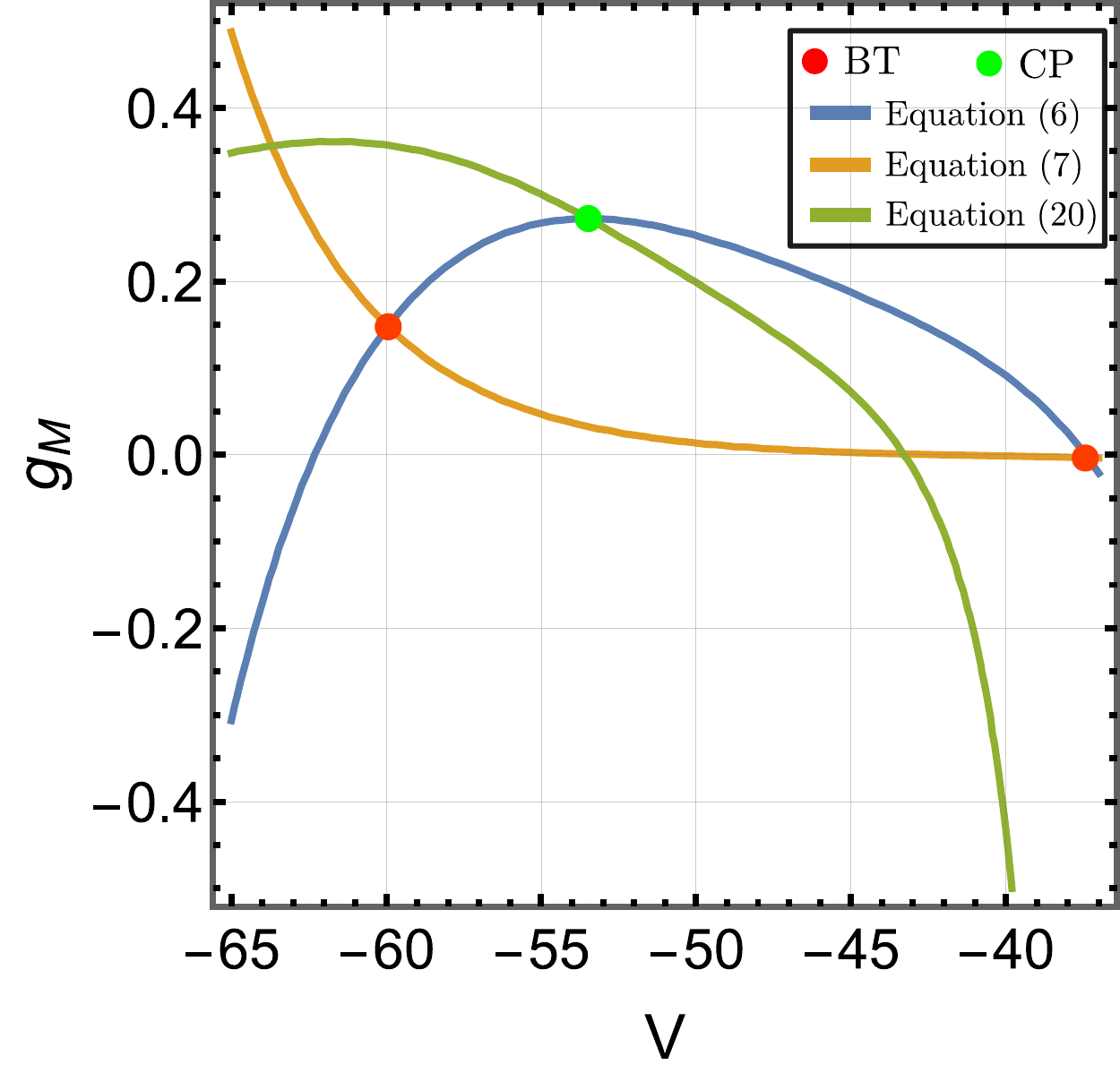}
 \caption{BT and CP points}
 \label{fig:Stiefel_BT}
\end{subfigure}
\begin{subfigure}{.46\textwidth}
  \centering
  \includegraphics[width=0.92\linewidth]{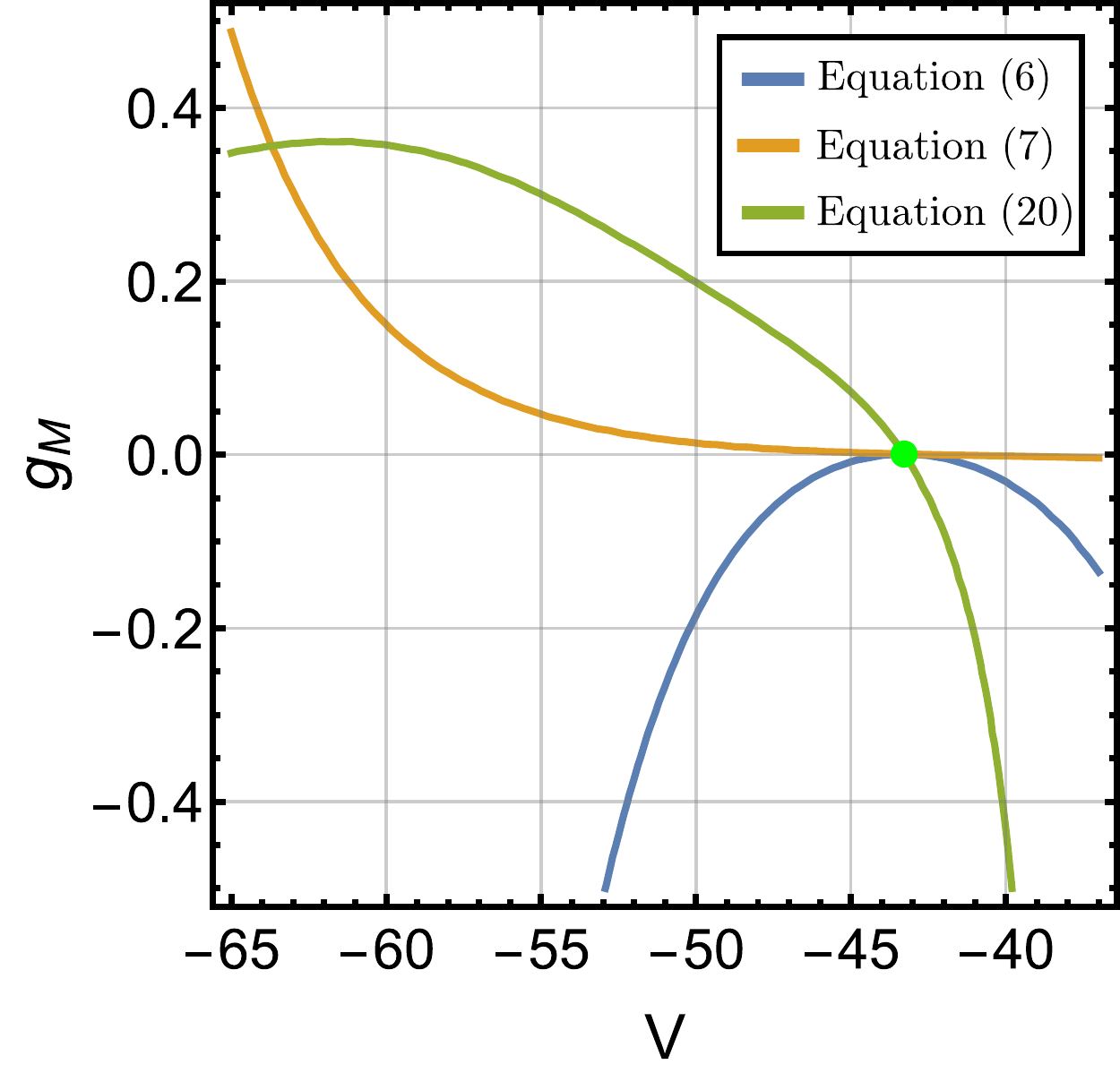}
 \caption{BTC point}
   \label{fig:Stiefel_BTC_A}
\end{subfigure}
\caption{Existence of codimension two and three bifurcation points in the Stiefel model \eqref{Stiefel_model}, with the parameter values given in Table \ref{Stiefel_parameters}. (a) The conditions given by equations~ \eqref{BT:Con_A1}, \eqref{BT:Con_A2} and \eqref{eq:dd_I_infter} are plotted in the $V,g_M$ space. The two intersection points (red dots) of the conditions in Theorem~\ref{Th:BTpoint} show that there are two BT points in the model. 
The one intersection point (green dot) of the conditions in Theorem \ref{Th:BTCpoint} shows the existence of one Cusp point; (b) The three conditions are plotted when the leak conductance is increased to $g_L=0.3785$. The intersection point (green dot) corresponds to the BTC point.}
\label{fig:Stiefel_BT_CP}
\end{figure}

Solving equations (\ref{BT:Con_A1}), (\ref{BT:Con_A2}) and (\ref{eq:dd_I_infter}) leads to the BT  point $(V^*,I_{app}^*,g_M^*)=(-59.9344,-0.0707,0.1482)$ and  Cusp point  $(\widehat{V},\widehat{I}_{app},\widehat{g}_M)=(-53.4754,0.0216,0.2724)$, see Figures \ref{fig:Stiefel_BT}. A second BT point occurs for $g_M<0$. These results are consistent with those found in \textsf{MATCONT}. Applying the analysis of section~\ref{sec:exist} to this model also shows that increasing $g_L$ should lead to a BTC point.
This is confirmed in Figure \ref{fig:Stiefel_BTC_A}.
We find the  BTC  point $(-43.1385,2.9461,0.0008)$ when we increase $g_L$ to $0.3785$.
As in the previous example, the  neuronal excitability type switches from Class-I to II as the conductances of the M-current increases, Class-I when $g_M<g_M^*$ and Class-II otherwise, see Figures  \ref{fig:Stiefel_model_gMIapp}, \ref{fig:Stiefel_model_V_Iapp} and \ref{fig:Stiefel_model_FI_Im1}. 
Although the range $(g_M^*,\widehat{g}_M)$  is much smaller than for Example 1, the model (\ref{Stiefel_model}) exhibits a similar behaviour in this range, see Figures   \ref{fig:WB_model_FI_Im2}.

\begin{figure}[htb!]
\centering
\begin{subfigure}{.40\textwidth}
  \centering
  \includegraphics[width=0.98\linewidth]{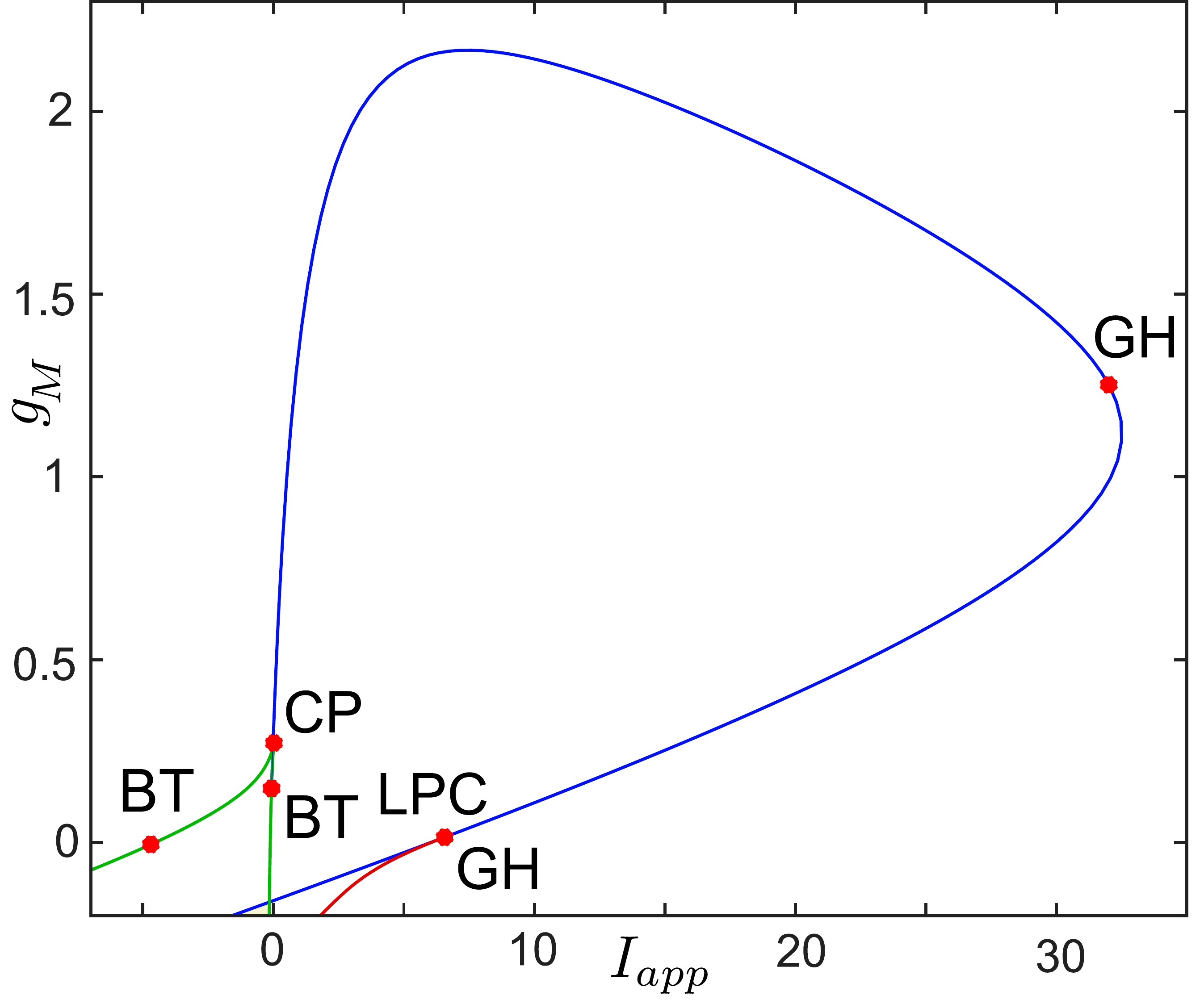}
    \caption{}
    \label{fig:Stiefel_model_gMIapp_a}
\end{subfigure}
\begin{subfigure}{.40\textwidth}
  \centering
  \includegraphics[width=0.95\linewidth]{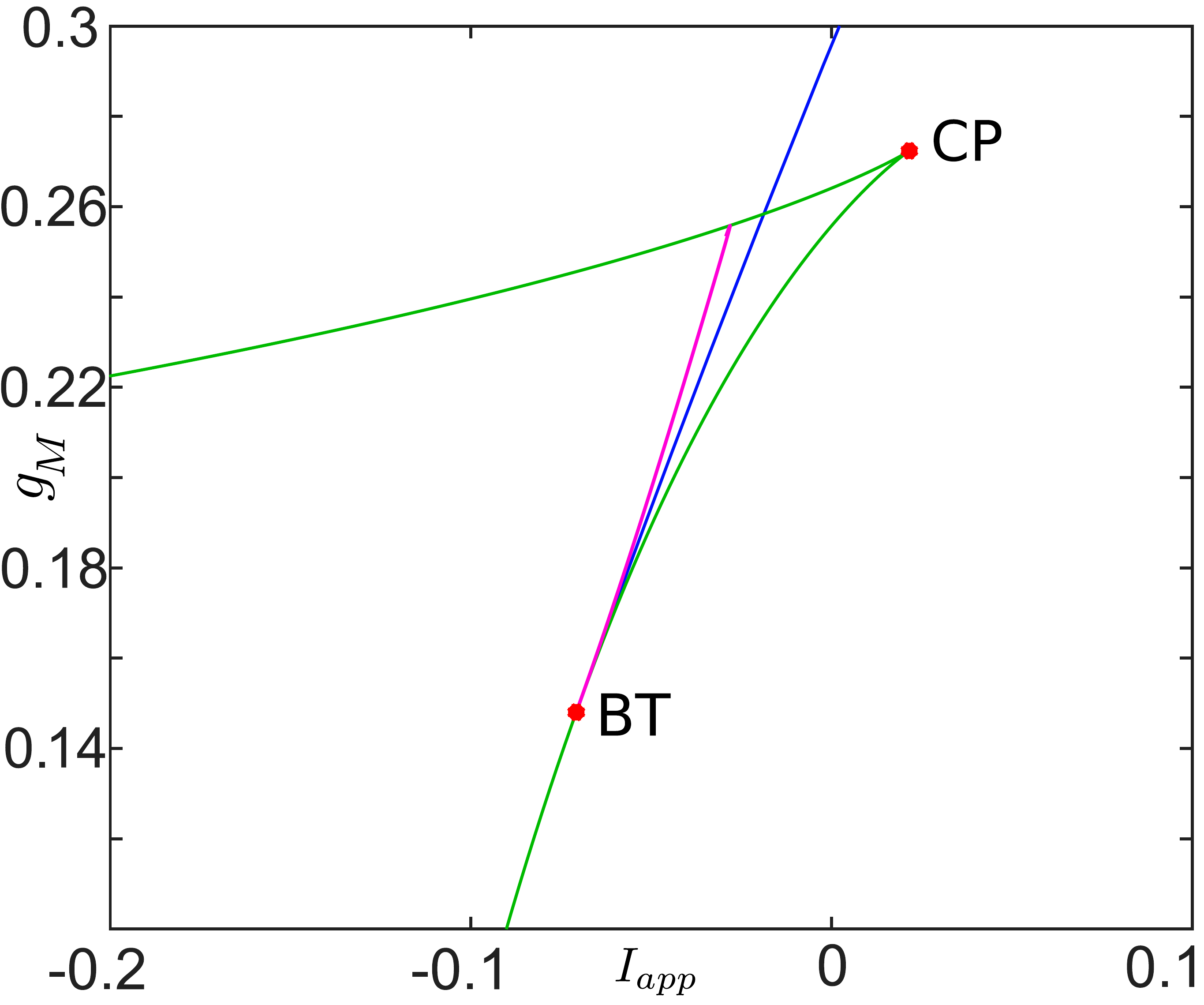}
   \caption{}
      \label{fig:Stiefel_model_gMIapp_b}
\end{subfigure}
\caption{Bifurcation diagram in $I_{app},g_M$ parameter space for Stiefel model  (\ref{Stiefel_model}). Green curves are limit point (fold/saddle-node) bifurcations of equilibria, blue are Andronov-Hopf bifurcations, magenta are homoclinic bifurcations and red  are limit point (fold) bifurcations of limit cycles (LPC). Co-dimension two bifurcation point labels are described in Table~\ref{Table:labels}.
}
\label{fig:Stiefel_model_gMIapp}
\end{figure}

 \begin{figure}[htb!]
\centering
\begin{subfigure}{.315\textwidth}
  \centering
  \includegraphics[width=0.95\linewidth]{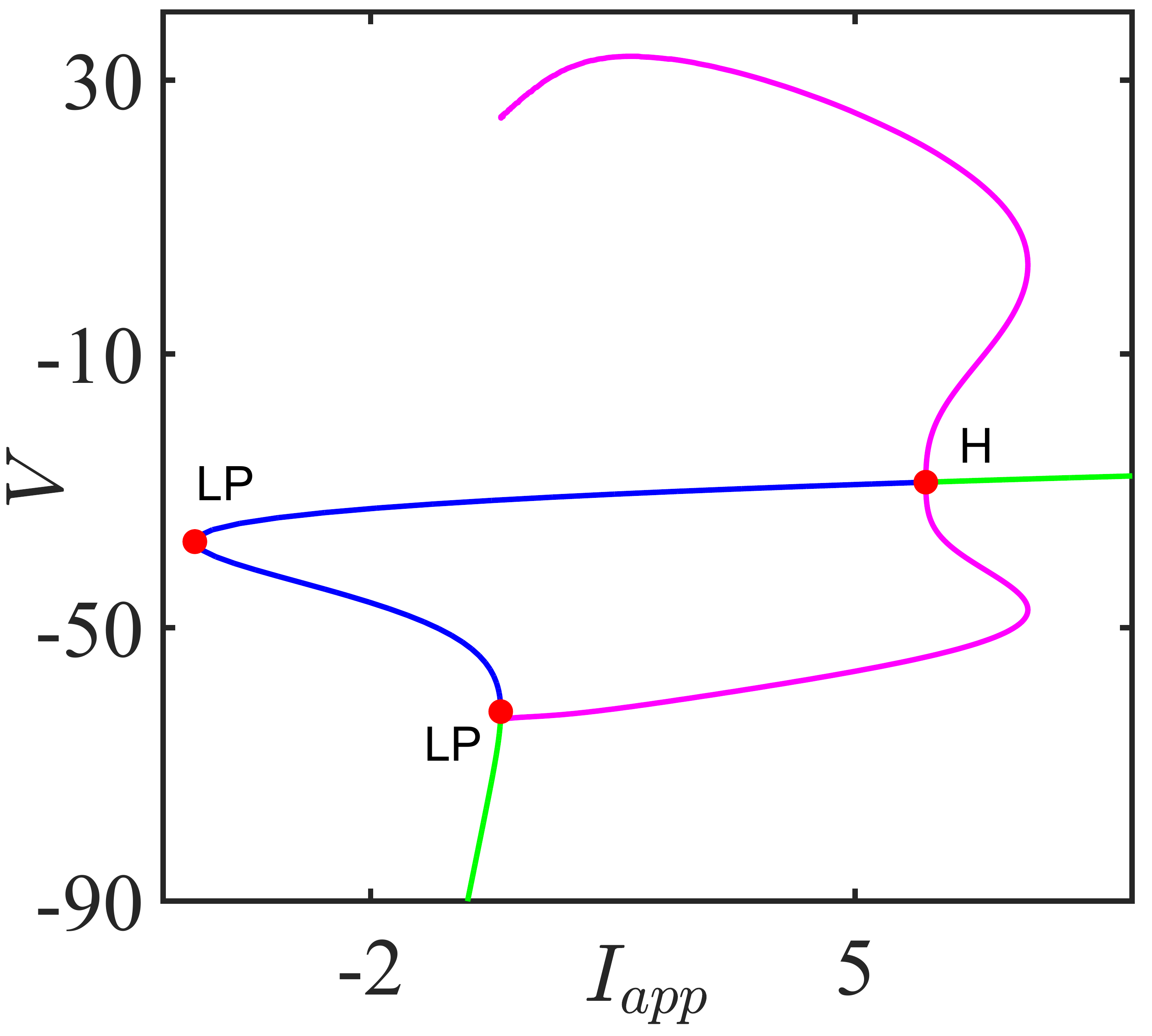}
    \caption{$g_M=0$}
    \label{fig:Stiefel_model_V_Iapp_a}
\end{subfigure}
\begin{subfigure}{.315\textwidth}
  \centering
  \includegraphics[width=0.95\linewidth]{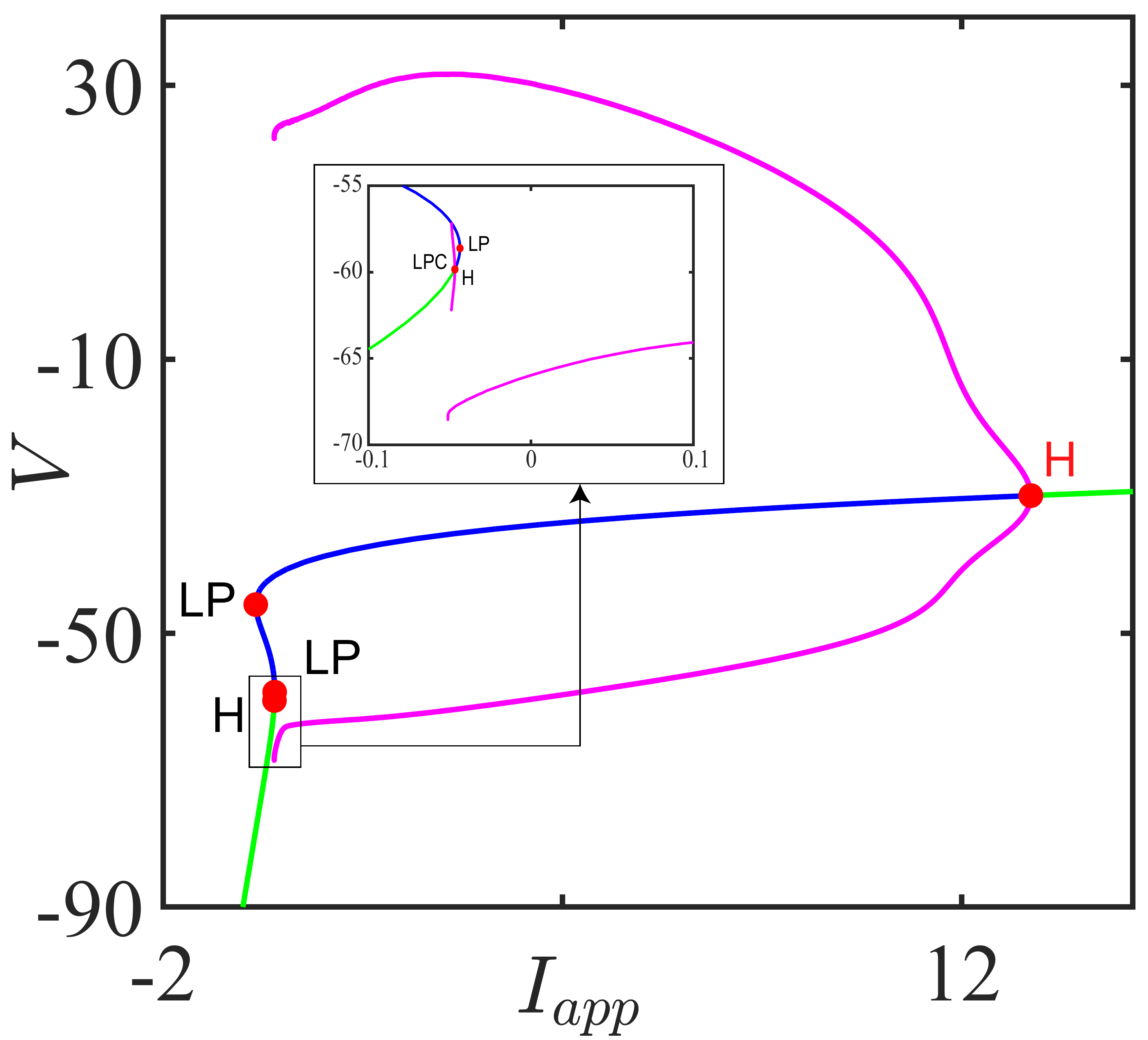}
  \caption{$g_M=0.2$}
      \label{fig:Stiefel_model_V_Iapp_b}
\end{subfigure}
\begin{subfigure}{.315\textwidth}
  \centering
  \includegraphics[width=0.95\linewidth]{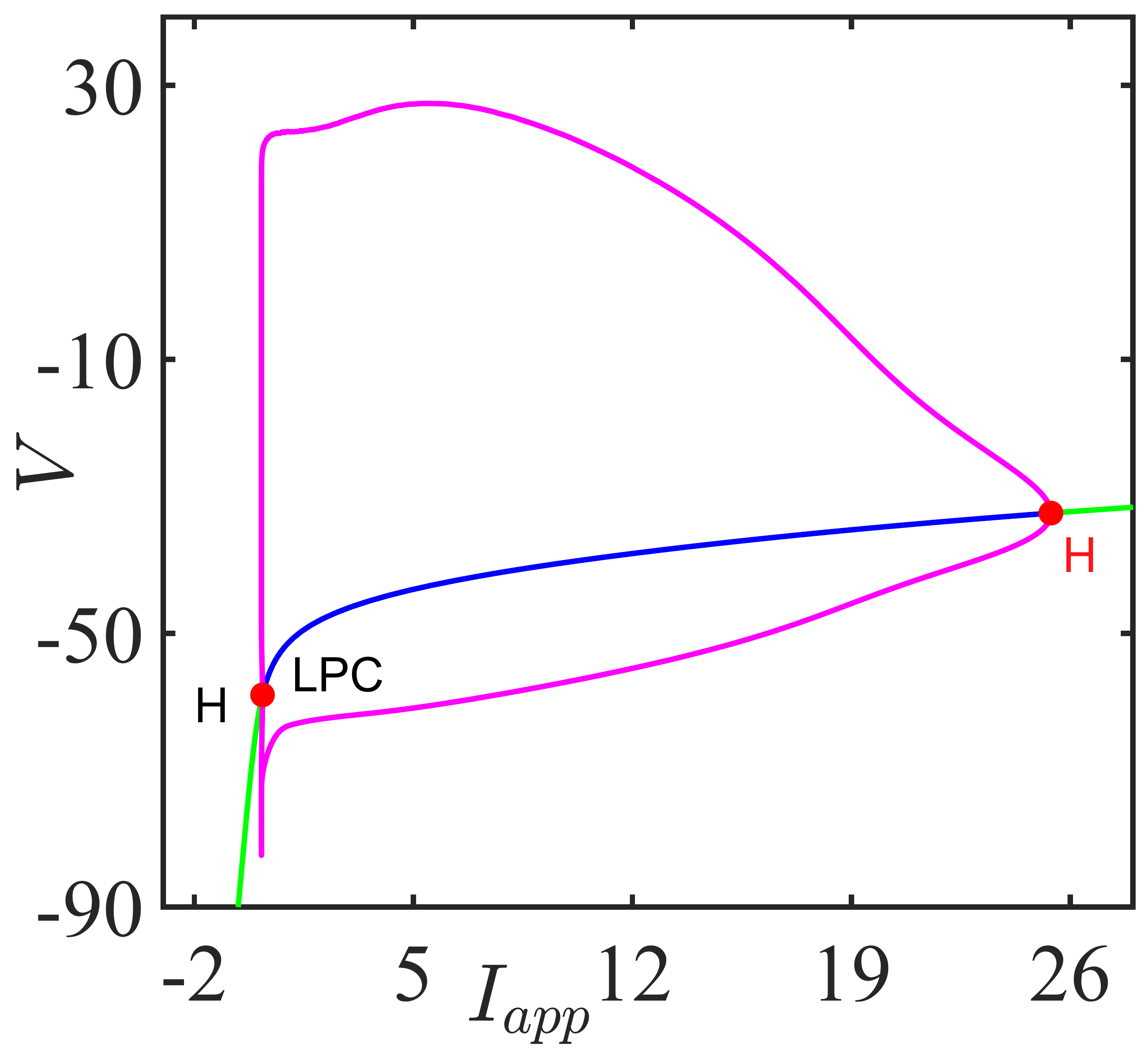}
  \caption{$g_M=0.6$}
      \label{fig:Stiefel_model_V_Iapp_c}
\end{subfigure}

\caption{One parameter bifurcation diagrams for the Stiefel model  (\ref{Stiefel_model}), showing the change in bifurcation structure as $g_M$ is varied. (a) $g_M<g_M^*$ (the value at the BT point); (b) $g_M^*<g_M<\widehat{g}_M$;
(c) $g_M>\widehat{g}_M$ (the value at the CP point).  Green/blue curves show stable/unstable equilibria. Pink curves show maxima/minima of periodic orbits. Co-dimension one bifurcation point labels are described in Table~\ref{Table:labels}.
}
\label{fig:Stiefel_model_V_Iapp}
\end{figure}

\begin{figure}[htb!]
\centering
  \includegraphics[width=0.95\linewidth]{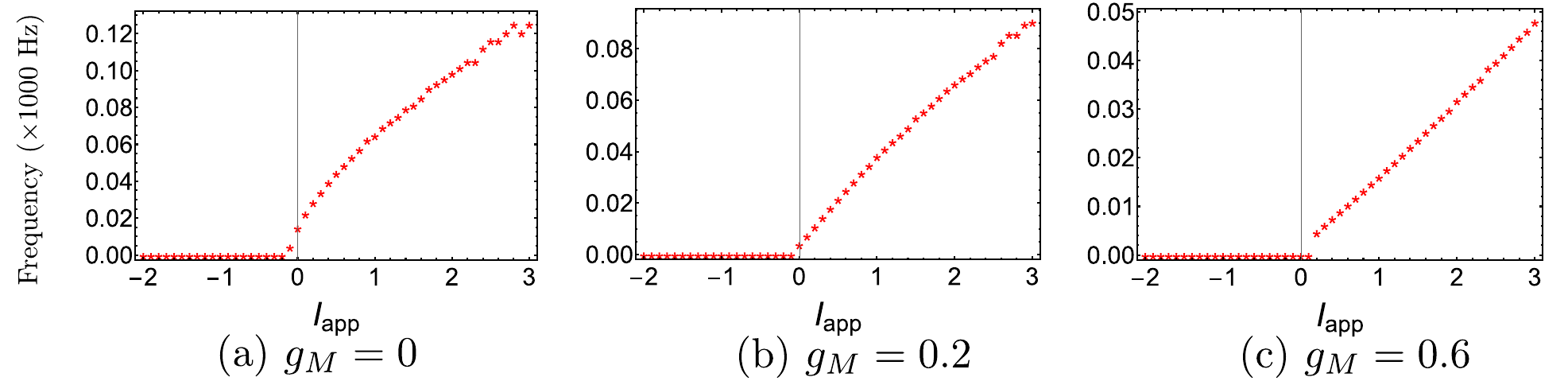}
   \caption{${ \rm F/I}$ curves of Stiefel model   (\ref{Stiefel_model}) corresponding to Figure~\ref{fig:Stiefel_model_V_Iapp}. (a) $g_M<g_M^*$ (the value at the BT point); (b) $g_M^*<g_M<\widehat{g}_M$;
(c) $g_M>\widehat{g}_M$ (the value at the CP point).}
    \label{fig:Stiefel_model_FI_Im1}
\end{figure}

\pdfbookmark[subsection]{Example 3: Reduced Traub-Miles model}{Example 3: Reduced Traub-Miles model}
\subsection*{Example 3.}  The Reduced Traub-Miles  (RTM) Model is a substantial simplification of a model of a pyramidal excitatory cell in rat hippocampus due to Traub and Miles \cite{traub1991neuronal}.
The RTM model with the M-current can be written as \cite{olufsen2003new} 
\begin{align}
	\label{RTM_model}
	C_m\frac{{dV}}{{dt}} &= {I_{app}} - {g_L}(V - {V_L}) - {g_M}w(V - {V_K}) - {g_{Na}}m^3 h(V - {V_{Na}})\nonumber\\
	&~ - {g_K}{n^4}(V - {V_K}),\\
	\frac{{d\sigma}}{{dt}} &= \frac{1}{{{\tau _{\sigma}}(V)}}\left( {{{\sigma}}_\infty }(V) - \sigma \right),\quad \sigma\in\{w,h,n,m\}.\nonumber
\end{align}
	Parameter values and other details are given in the Appendix.

 Both the analytical results and  \textsf{MATCONT} give the bio-physically permissible BT point $(V^*,I_{app}^*,g_M^*)=(-63.7386,0.2449,.0659)$ and Cusp point  $(\widehat{V},\widehat{I}_{app},\widehat{g}_M)=(-50.8204,71.9395,14.5123)$, see  Figure~\ref{fig:RTM_BT}. 
Applying the analysis of section~\ref{sec:exist} again shows that increasing $g_L$ should lead to a BTC point.
This is confirmed in Figure \ref{fig:RTM_BTC_A}. When we increase $g_L$ to $13.79$, the BT and CP points collide producing the BTC point  $(-49.8762,166.25,-0.6745)$. 
 In this example, we notice that the range $(g_M^*,\widehat{g}_M)$  is bigger than those   in Example 1 and 2 but  
 the transition in the neuronal excitability type is consistent with previous examples:  Class-I when $g_M<g_M^*$ and Class-II otherwise, see Figures  \ref{fig:RTM_model_gMIapp} and \ref{fig:RTM_model__V_Iapp}.

 \begin{figure}[htb!]
\centering
\begin{subfigure}{.45\textwidth}
  \centering
  \includegraphics[width=0.95\linewidth]{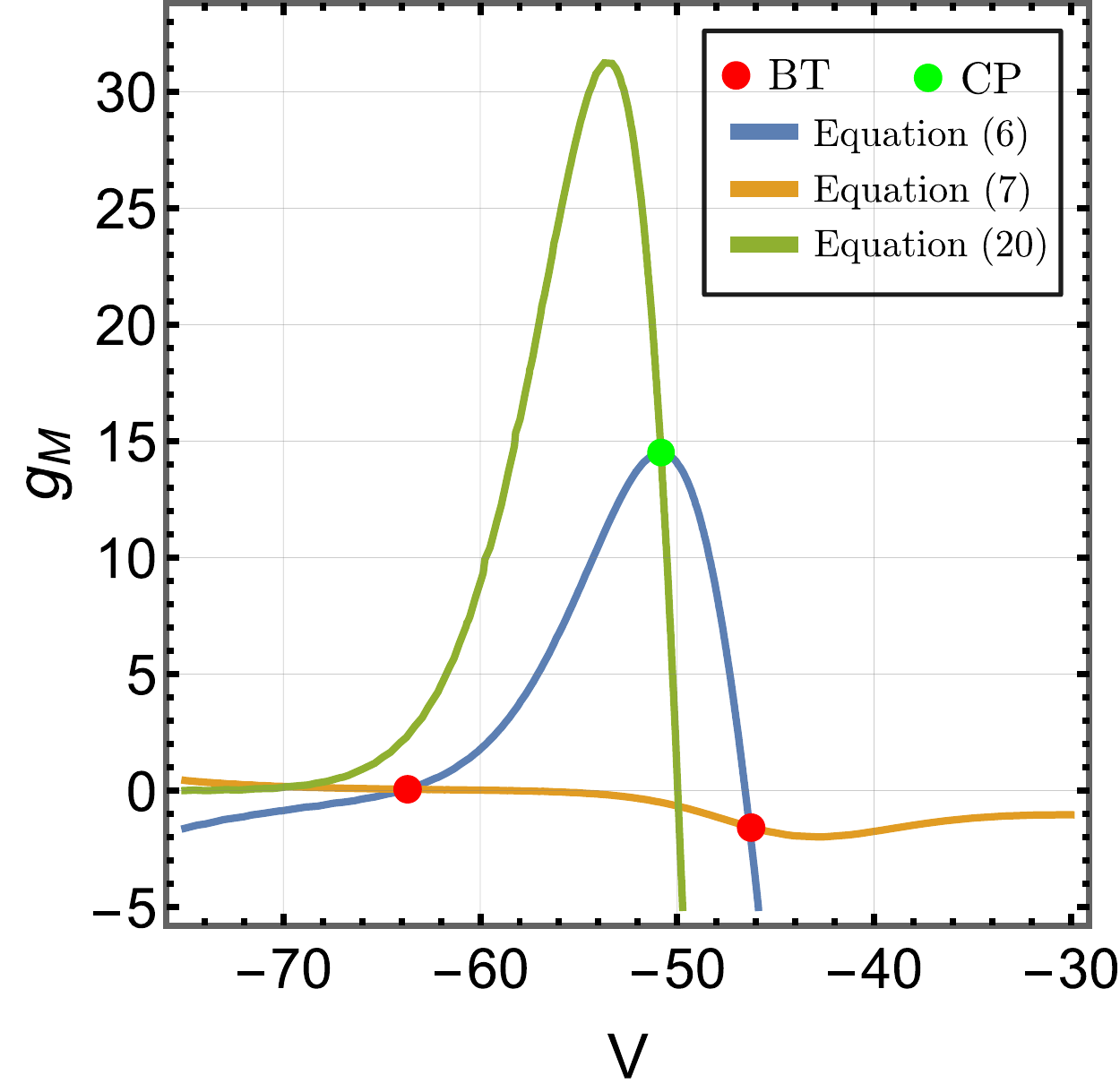}
 \caption{BT and CP points}
 \label{fig:RTM_BT}
\end{subfigure}
\begin{subfigure}{.45\textwidth}
  \centering
  \includegraphics[width=0.95\linewidth]{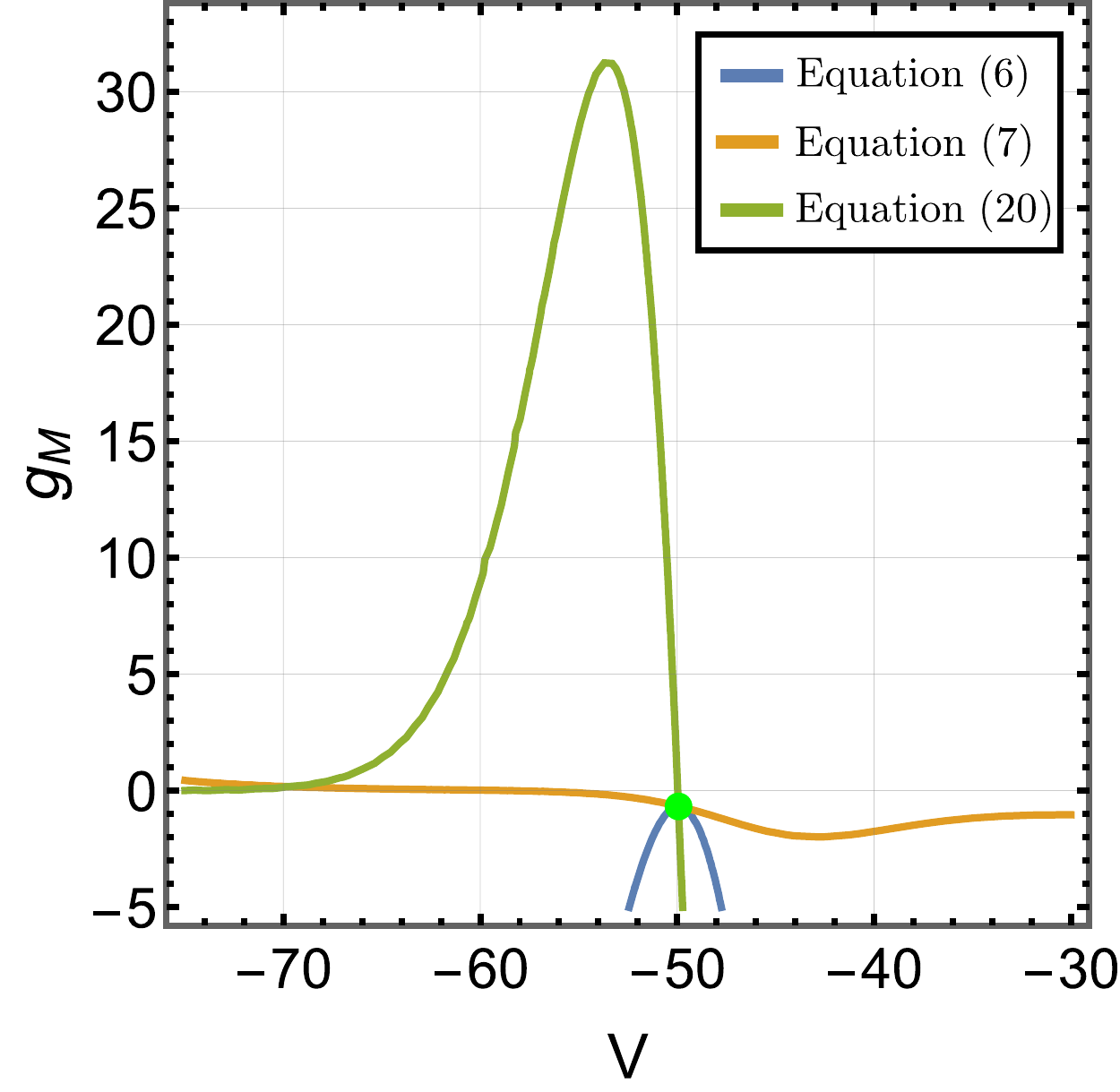}
 \caption{BTC point}
   \label{fig:RTM_BTC_A}
\end{subfigure}
\caption{Existence of codimension two and three bifurcation points in the Wang-Buz\'saki model (\ref{RTM_model}), with the parameter values given in Table \ref{RTM_parameters}. (a) The conditions given by equations~ \eqref{BT:Con_A1}, \eqref{BT:Con_A2} and \eqref{eq:dd_I_infter} are plotted in the $V,g_M$ space. The two intersection points (red dots) of the conditions in Theorem~\ref{Th:BTpoint} show that there are two BT points in the model. 
The one intersection point (green dot) of the conditions in Theorem \ref{Th:BTCpoint} shows the existence of one Cusp point; (b) The three conditions are plotted when the leak conductance is increased to $g_L=13.79$. The intersection point (green dot) corresponds to the BTC point.}
\label{fig:RTM_BT_CP}
\end{figure}

\begin{figure}[htb!]
\centering
\begin{subfigure}{.40\textwidth}
  \centering
  \includegraphics[width=0.95\linewidth]{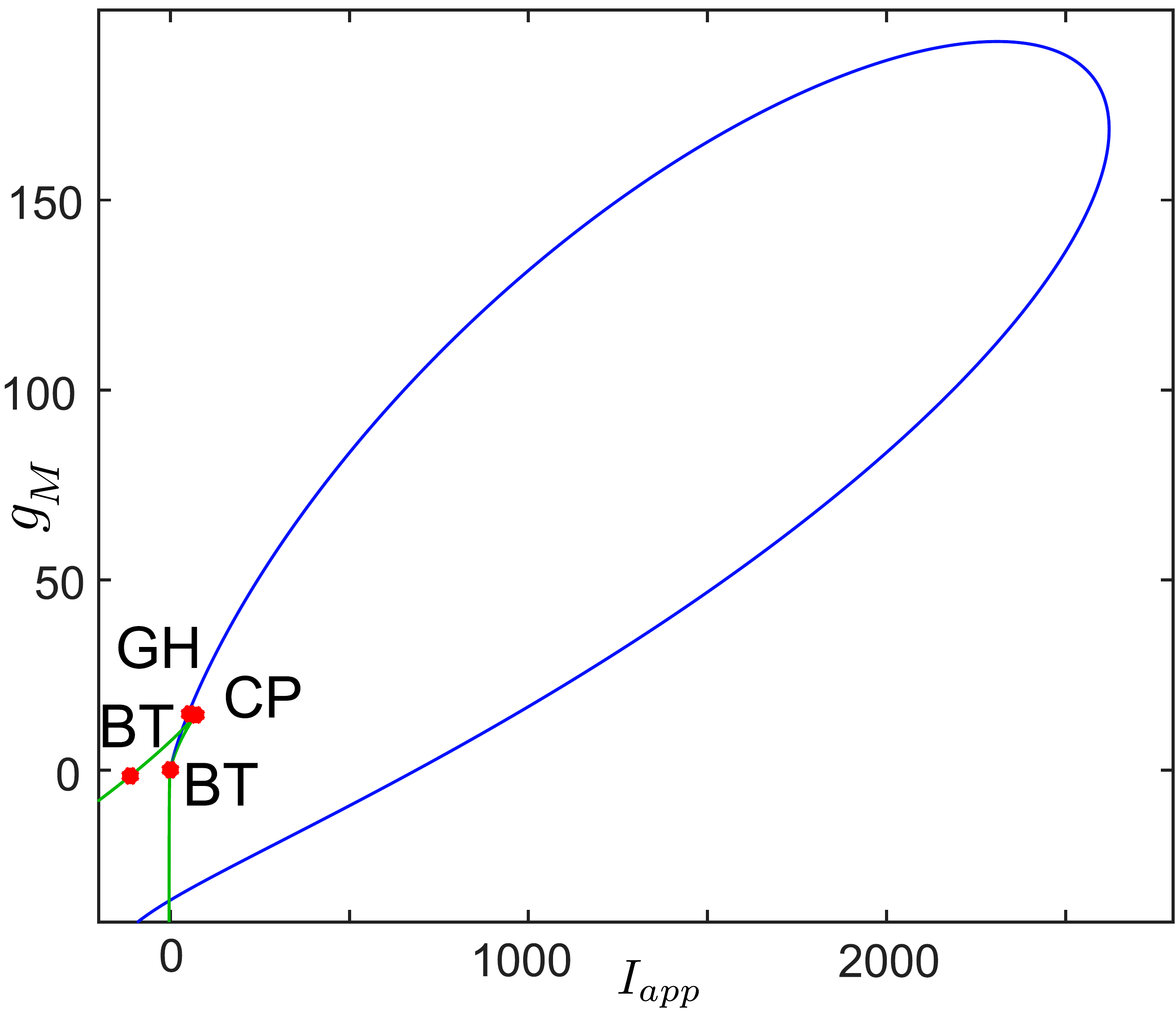}
    \caption{}
    \label{fig:RTM_model_gMIapp_a}
\end{subfigure}
\begin{subfigure}{.40\textwidth}
  \centering
  \includegraphics[width=0.95\linewidth]{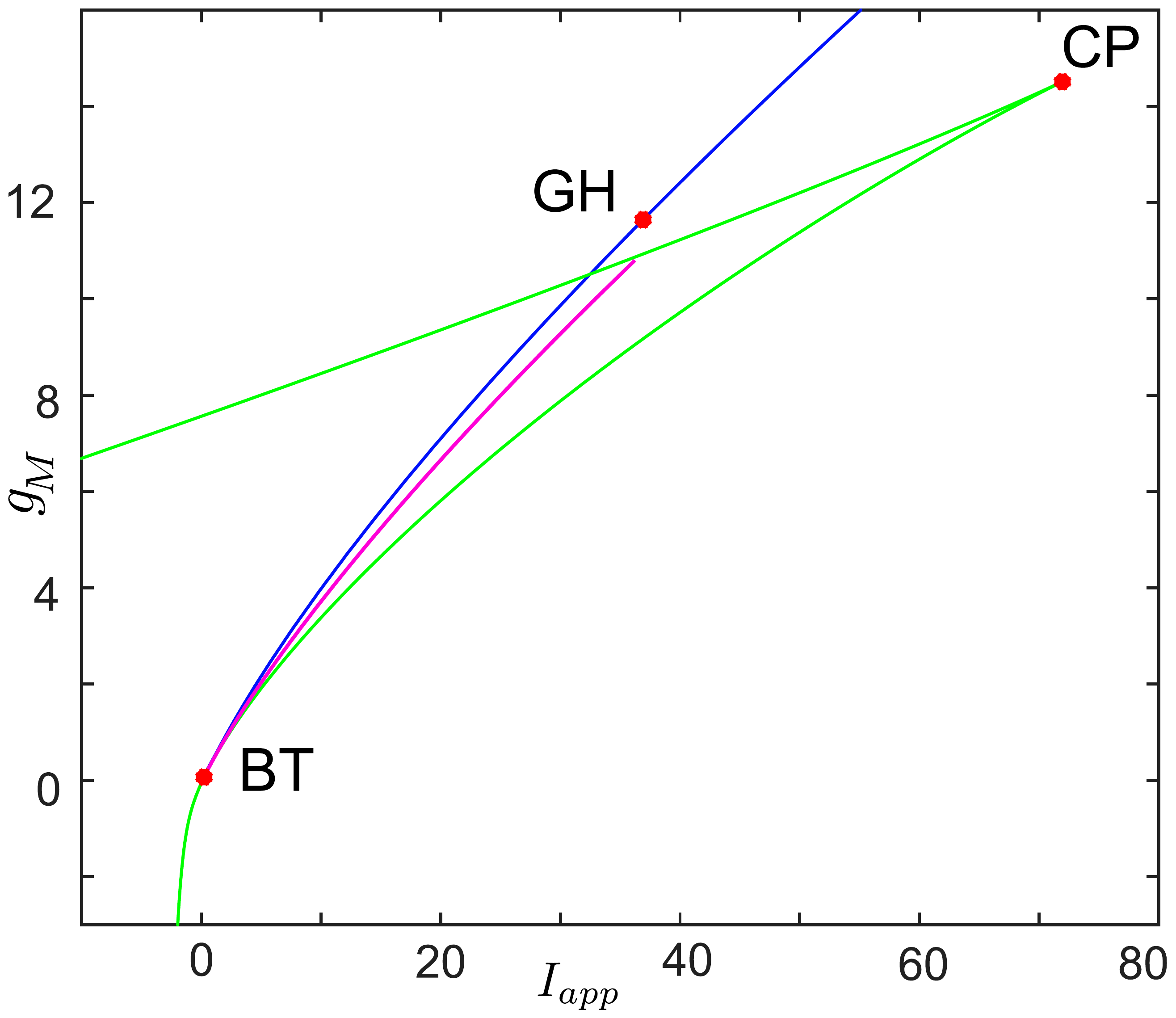}
   \caption{}
       \label{fig:RTM_model_gMIapp_b}
\end{subfigure}
\caption{Bifurcation diagram in $I_{app},g_M$ parameter space for the RTM model  (\ref{RTM_model}). Green curves are limit point (fold/saddle-node) bifurcations of equilibria, blue are Andronov-Hopf bifurcations, magenta are homoclinic bifurcations and red  are limit point (fold) bifurcations of limit cycles (LPC). Co-dimension two bifurcation point labels are described in Table~\ref{Table:labels}.}
\label{fig:RTM_model_gMIapp}
\end{figure}

 \begin{figure}[htb!]
\centering
\begin{subfigure}{.315\textwidth}
  \centering
  \includegraphics[width=0.95\linewidth]{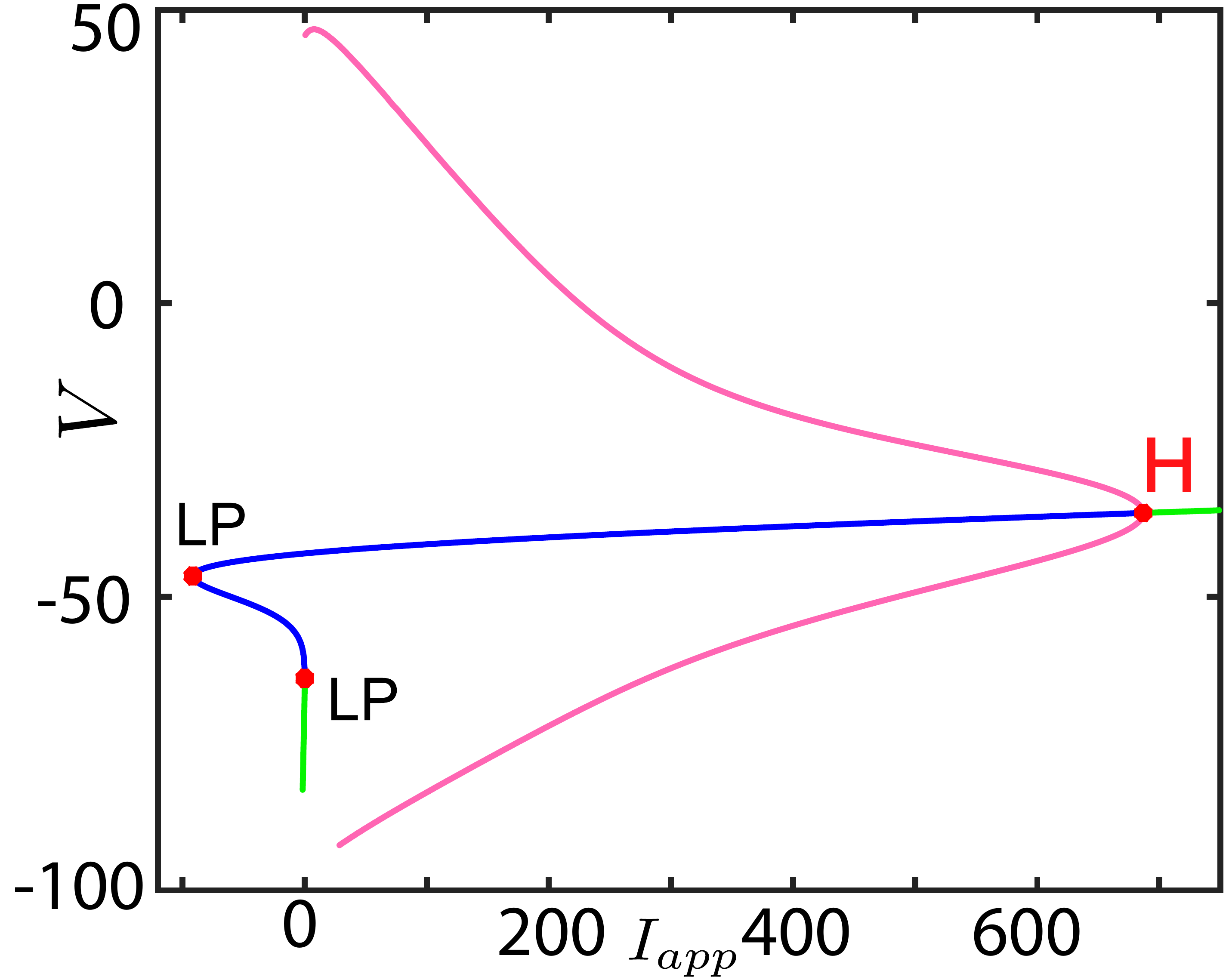}
    \caption{$g_M=0$}
    \label{fig:RTM_model_V_Iapp_a}
\end{subfigure}
\begin{subfigure}{.315\textwidth}
  \centering
  \includegraphics[width=0.95\linewidth]{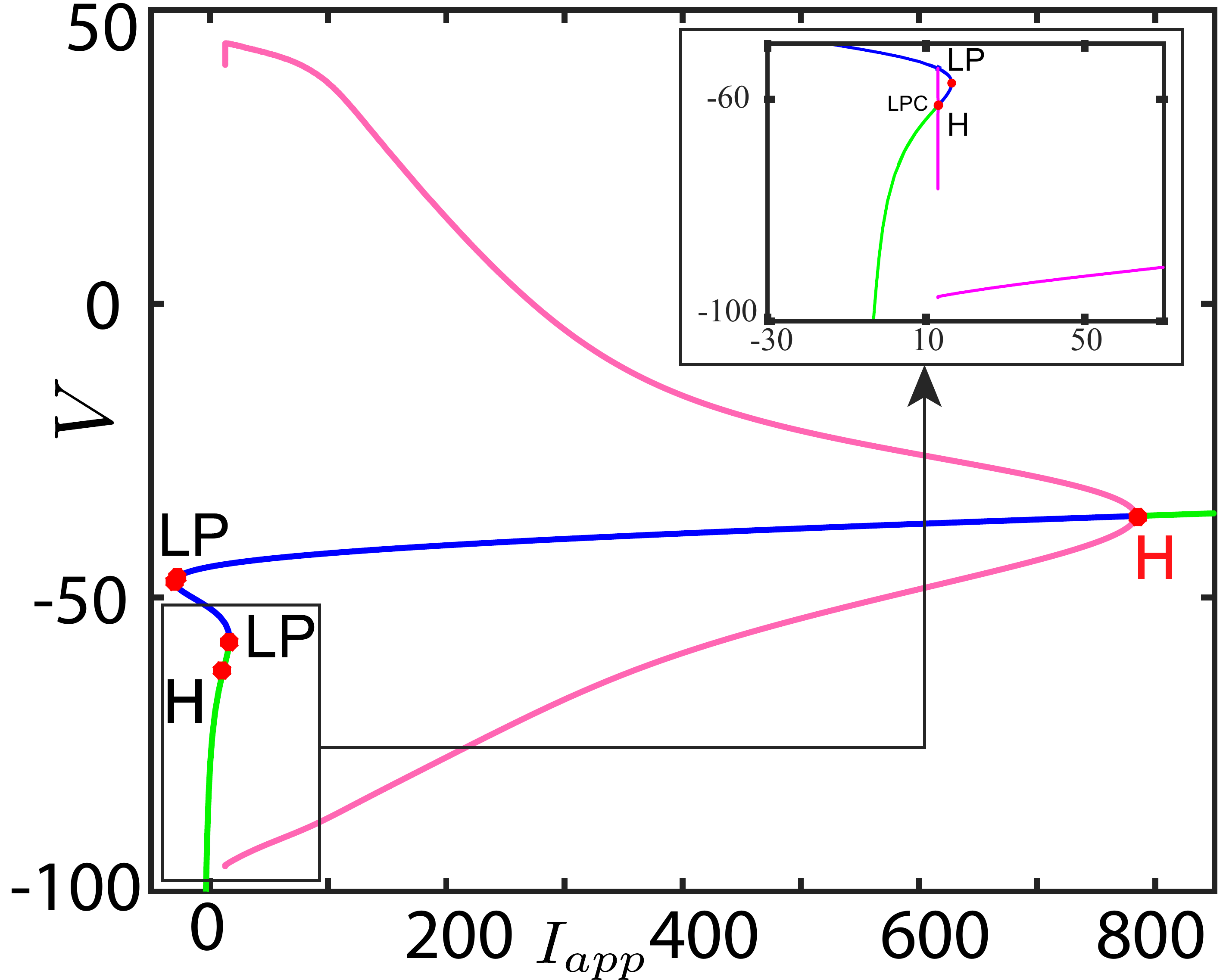}
   \caption{$g_M=5$}
       \label{fig:RTM_model_V_Iapp_b}
\end{subfigure}
\begin{subfigure}{.315\textwidth}
  \centering
  \includegraphics[width=0.95\linewidth]{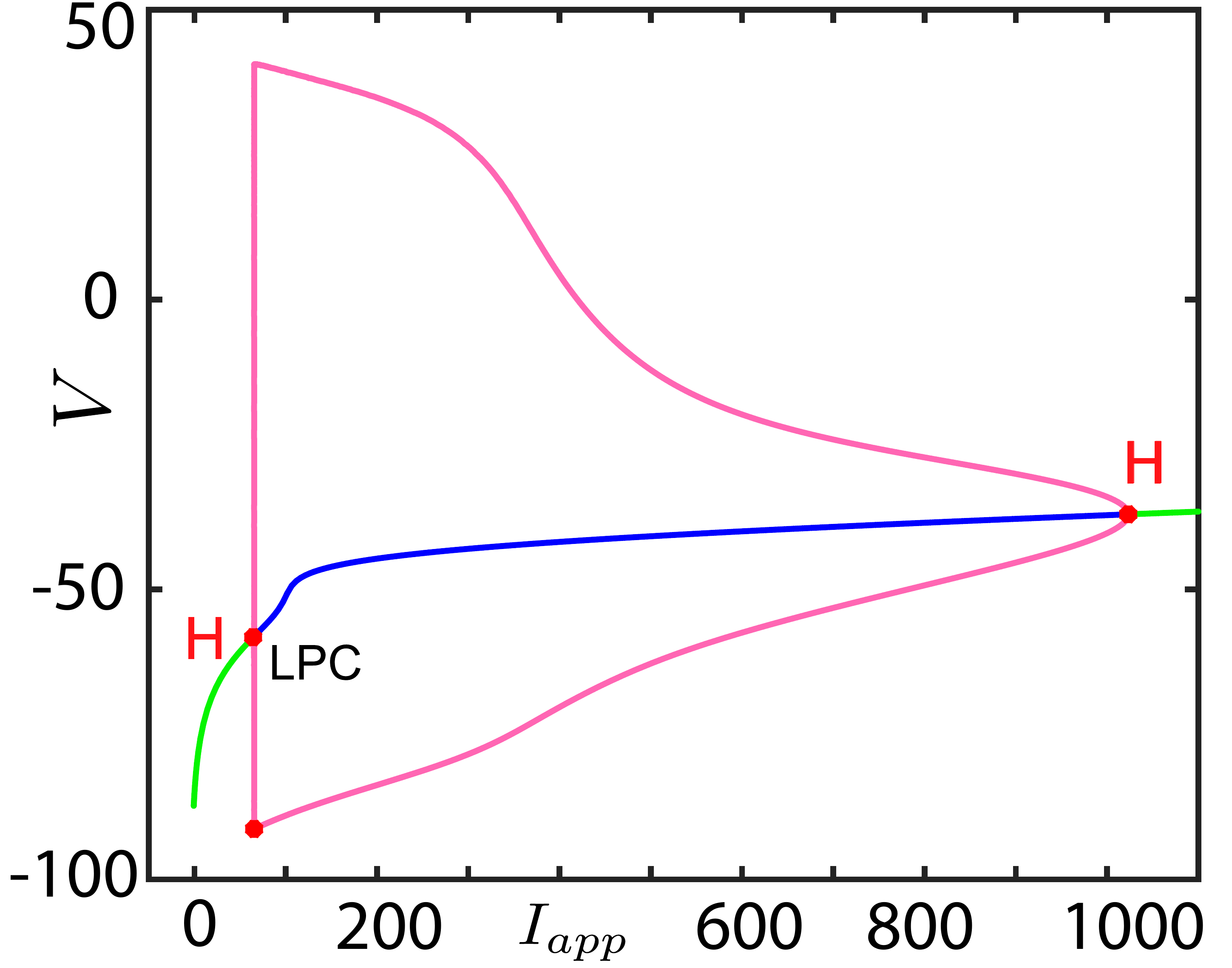}
   \caption{$g_M=18$}
       \label{fig:RTM_model_V_Iapp_c}
\end{subfigure}
\caption{One parameter bifurcation diagrams for the RTM model    (\ref{RTM_model}), showing the change in bifurcation structure as $g_M$ is varied. (a) $g_M<g_M^*$ (the value at the BT point); (b) $g_M^*<g_M<\widehat{g}_M$;
(c) $g_M>\widehat{g}_M$ (the value at the CP point).  Green/blue curves show stable/unstable equilibria. Pink curves show maxima/minima of periodic orbits. Co-dimension one bifurcation point labels are described in Table~\ref{Table:labels}.
}
\label{fig:RTM_model__V_Iapp}
\end{figure}

\section{Implications for Synchronization}
\label{sec:Implications}
In section~\ref{sec:exist} we have shown that the M-current will give rise to a BT bifurcation in any conductance-based neural model, when certain conditions are met. In  section~\ref{sec_Numerical} we showed in three examples that these conditions are met and a BT bifurcation occurs. Further, we showed that
this BT bifurcation induces a transition from Class-II to Class-I excitability in these models as the conductance of the M-current is decreased (as would be the case in the presence of acetylcholine).
In this section we will explore one implication of this transition.
There are many studies in the literature describing the relationship between the synchronization of coupled neurons and their neuronal excitability type, see e.g., \cite{ermentrout1996type,hansel1995synchrony}.  
The classic result is that the in-phase solution of a pair of weakly coupled Class-I oscillators model 
with synaptic coupling is stable when there are inhibitory coupling and unstable for excitatory 
coupling, while the anti-phase solution exhibits the opposite stability \cite{ermentrout1996type}. The synchronization of Class-II oscillators is less clear, and other factors such as the synaptic
time constants and firing frequency may affect these conclusions \cite{ermentrout1996type,stiefel2009effects}.
By in-phase solution, we mean both oscillators reach their highest peak at the same time,  whereas an anti-phase solution means one oscillator reaches its highest peak one half-period after the other oscillator. 

To study the stability of phase-locked solutions and the correspondence with the neuronal excitability type as $g_M$ varies, we write two coupled neurons  with synaptic coupling  as 
 \begin{neweq}
\label{coupled_N}
{C_m}\frac{{d{V_i}}}{{dt}} &= {I_{app}} - {g_L}({V_i} - {V_L}) - {g_M}w({V_i} - {V_K}) -I_{ion}(V)- {g_{syn}}{s_j}({V_i} - {V_{syn}})),\\
\frac{{dw}}{{dt}} &= \frac{{{1}}}{{{\tau _w}(V)}}\left( {{w_\infty }({V_i}) - w} \right),\\
\frac{{d{s_i}}}{{dt}} &={a_{{e_0}}} {a_e}(V) (1 - {s_i})-\frac{{  {s_i}}}{{{\tau _s}}},
 \end{neweq}
for $i,j=1,2$ such that $i\ne j$, where $I_{ion}$ are ionic currents in Examples 1-3. The synaptic coupling function and parameters are given Table \ref{Table:synaptic}. 
 
\begin{table}[hbt!]
\centering
\begin{tabular}{|l|c|c|c|c|}
\hline
          & ${a_{{e_0}}}$ & ${\tau _s}$  & ${a_e}(V)$& Reference \\ \hline
Example 1: $V_{syn}=0,-75$ & $6.25$ & $5$ & ${{{\left( {1 + \exp \left( {\frac{{ - V}}{2}} \right)} \right)}^{ - 1}}} $&\cite{skinner2005two}\\ \hline
Example 2: $V_{syn}=0,-80$  & $4$ & $8$ & ${{{\left( {1 + \exp \left( {\frac{{ - V}}{5}} \right)} \right)}^{ - 1}}} $ &\cite{stiefel2008cholinergic}\\ \hline
Example 3: $V_{syn}=0$& $5$ & $2$ & $(1 + \tanh(V/4))$& \cite{olufsen2003new}\\ \hline
Example 3: $V_{syn}=-80$& $2$ & $10$ & $(1 + \tanh(V/4))$ &\cite{olufsen2003new}\\ \hline
\end{tabular}
\caption{Synaptic coupling function and parameters in (\ref{coupled_N}).}
\label{Table:synaptic}
\end{table}

To determine  the stable phase-locked solution(s), first, we solve   (\ref{coupled_N}) numerically with ten random initial conditions at each step of $g_M$ then we calculate the period of the oscillators  ($T_1$ and $T_2$) in the numerical solution. Finally, we  approximate the phase shift as  
\beqn
\varphi  = 2\pi \left( {\frac{\tau }{\mathcal{T}} - \left\lfloor {\frac{\tau }{\mathcal{T}}} \right\rfloor } \right)
\label{appo}
\eeqn
where $\left\lfloor  \cdot  \right\rfloor$ is the floor function, $\mathcal{T}=(T_1+T_2)/2$ and $\tau$ is the argument shift satisfying $V_1(t)=V_2(t+\tau)$ for all $t$. 
Figure  \ref{fig:TWOneurons} shows a bifurcation diagrams for (\ref{coupled_N}) with excitatory and  inhibitory synaptic coupling in  Examples 1-3. 
For  instance, for coupled Wang-- Buzsaki model (Example 1), we notice in
Figure  \ref{fig:TWOneurons}a  that when $g_M<g_M^*$ (Class-I dynamics in (\ref{WB_model})), the in-phase solution is unstable  and the anti-phase solution is stable with excitatory coupling $V_{syn}=0$. 
The reverse is true for inhibitory coupling $V_{syn}=-75$.
This is consistent with \cite{ermentrout1996type}.
When there is an excitatory synaptic connection, as the M-current reaches $g_M\approx 0.5$, the anti-phase solution loses its stability  and two stable out-of-phase solutions (neither in-phase nor anti-phase) appear.  
As the conductance of the M-current  is increased any further, a stable in-phase solution appears. 
Hence, there is a transition from stable anti-phase solution to stable in-phase solution 
via stable out-of-phase solutions. 
The transition also occurs at $g_M\approx 0.5$ when there is the coupling is inhibitory.
We observe a similar dynamical behaviour in Examples 2 and 3, see Figure \ref{fig:TWOneurons}b-\ref{fig:TWOneurons}c, although the transition is not as clear in all cases. 
Although the relationship of the transition point to the codimension two bifurcations varies with
the different models, in all cases it occurs at some $g_M\in(g_M^*,\widehat{g}_M)$, that is, 
when the model has Class-II dynamics.

\begin{figure}[htb!]
  \centering
  \includegraphics[width=0.95\linewidth]{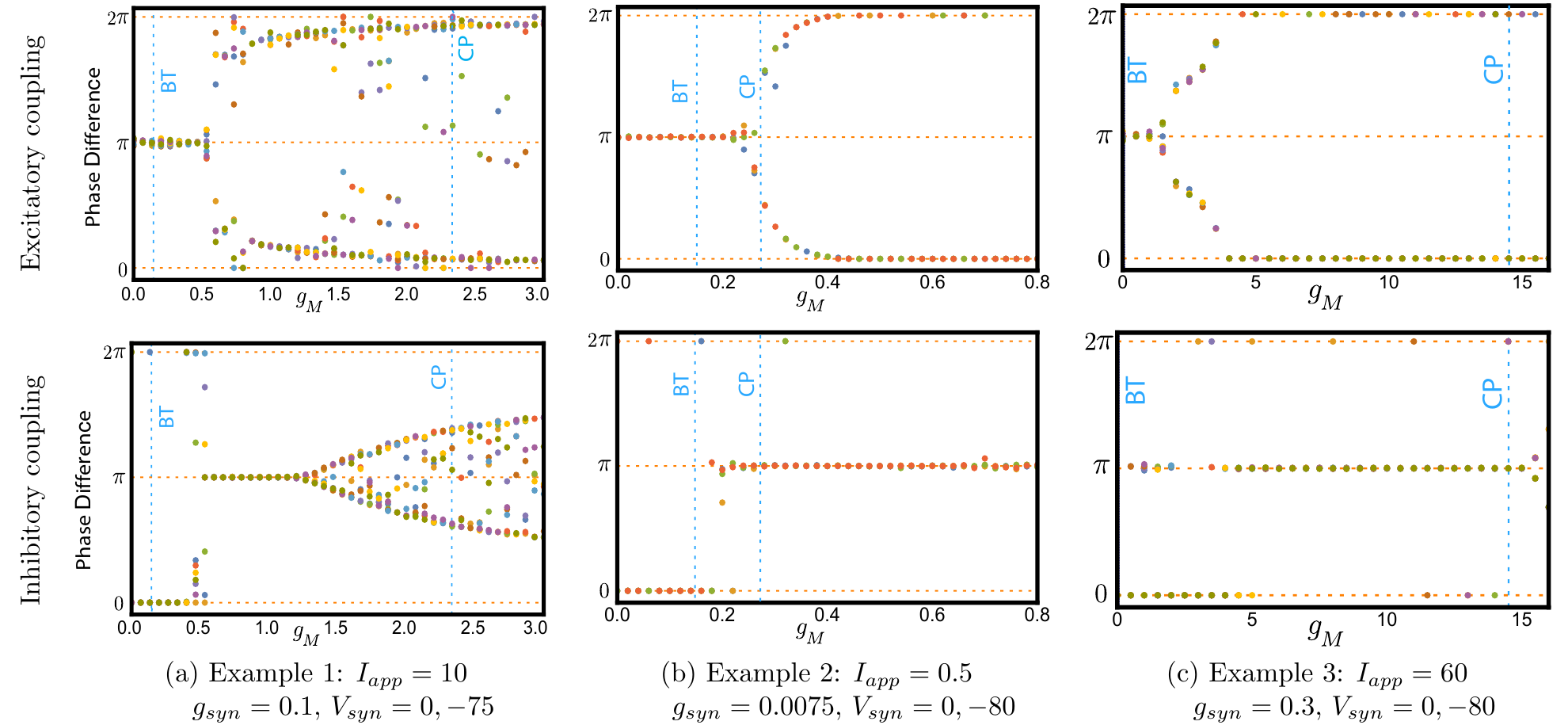}
   \caption{Bifurcation diagrams showing the change of synchronization of two identical, synaptically coupled neurons as $g_M$ is varied. For all examples, with $g_M<g_M^*$ (the BT point) excitatory coupling leads to phase-locking in anti-phase (phase difference $\pi$), while inhibitory coupling leads to in-phase (phase difference $0$). In all cases $g_M$ has to be increased significantly past the $BT$ value before the solution switches to in-phase for excitatory coupling and anti-phase for inhibitory coupling.}
 \label{fig:TWOneurons}
\end{figure}

As indicated above, other factors may affect the synchronization of neurons. 
We focus here on the firing frequency of the neuron.
In \cite{stiefel2009effects} it was shown that increasing the firing frequency 
by increasing the applied current could switch the PRC of a model neuron with an M-current
from Type-II to closer to Type-I.
In \cite{fink2011cellularly}, the authors reproduced this result for other neural models and studied how 
changing the firing frequency 
modulates the synchronization properties induced by the M-current. They found that synchrony in excitatory networks of neurons with a Type I PRC (low $g_M$) was largely unaffected by frequency modulation, whereas networks of Type II PRC neurons (high $g_M$) synchronized much better at lower frequencies. 
In \cite{ermentrout2010mathematical}, the authors  studied how the stability of in-phase and anti-phase phase-locked solutions in Wang--Buz\'saki model (with no M-current) varied with firing frequency.  
At low frequencies with inhibitory coupling, they showed that  both in-phase and anti-phase phase-locked solutions are stable. However, at higher frequencies only the in-phase solution  is stable. 
In contrast, with excitatory coupling, they showed that the in-phase solution  is unstable for both high and low frequencies. Recalling that the Wang--Buz\'saki model is a Class-I oscillator, this latter
result is consistent with that of \cite{fink2011cellularly}.

To consider if firing frequency has an effect in our results, we determined the variation of 
firing frequency with the conductance of the M-current, $g_M$, for our example models,
see Figure \ref{fig:model_F_g_M}. In all cases the firing frequency decreases rapidly as
$g_M$ increases. When the models are in the Class-I excitability regime (below the BT
point), the frequency
change does not affect the sychronization properties. This is consistent with the results described above~\cite{stiefel2009effects,fink2011cellularly}, given that neurons with Class-I excitability typically have Type-I PRCs~\cite{ermentrout1996type}.
Recalling that the main switch in synchronization behaviour in all cases occurs within 
the Class-II regime, we conclude that this switch is likely due to the decrease in the
frequency as $g_M$ increases.

In summary, while the excitability class of the model changes exactly at the BT point, the synchronization property of the models switches at $g_M$ value larger than the BT
point, when the frequency of the intrinsic oscillations is small enough.

\begin{figure}[htb!]
\centering
  \includegraphics[width=0.95\linewidth]{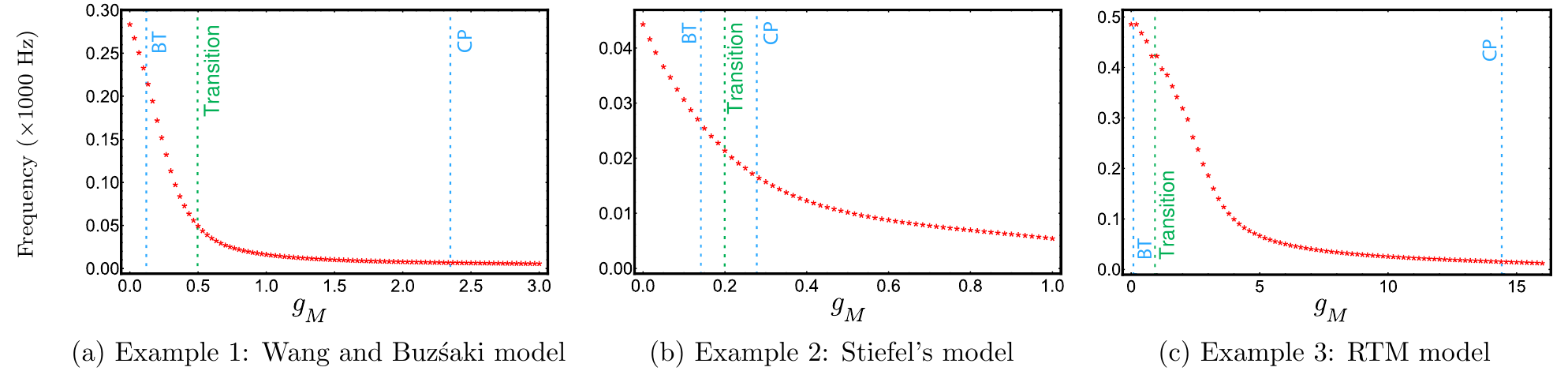}
    \caption{${\rm F/g_{M}}$ curve of the models in Examples 1-3. For each model neuron, the applied current was fixed at a value which yielded stable periodic solutions for all $g_M$ in the given range, then the frequency of the periodic solution was plotted against $g_M$. The blue dashed lines show the $g_M$ values corresponding to the BT and CP points. The green dashed line shows the $g_M$ value where the change in sychronization occurs for the coupled neurons in Figure~\ref{fig:TWOneurons}.}
\label{fig:model_F_g_M}
\end{figure}

 \section{Discussion}
 \label{sec_Conclusions}

In this paper, we studied Bogdanov-Takens (BT) bifurcation in  a general conductance-based neuron model with the inclusion of the M-current.
We started by showing the existence of equilibrium points. Then, we derived the necessary and sufficient conditions for the  equilibrium point to become a BT point. 
A degenerate Bogdanov-Takens (BTC) point appears when BT and Cusp points merge.
To discuss the occurrence of such point, we provided the condition for a Cusp bifurcation. We then showed that the conditions for the BT and Cusp bifurcation may be satisfied by varying the applied
current and the maximal conductance of the M-current and that for the BTC point by additionally varying the conductance of the leak current. 
 
As previously noted, our theoretical work was inspired by two recent papers. 
In \cite{kirst2015fundamental} they show that the BTC point can occur in any conductance-based model in the parameter space of the applied current, leak conductance and capacitance. They use this to study the effect of the leak current on the excitability properties of models for single neurons and synchronization properties for networks of neurons.
In \cite{pereira2015bogdanov} they study a general conductance-based neural model. They show that if the model has an equilibrium point with a double zero eigenvalue for some parameter values, then it is a BT point. Further, they give conditions on the gating variables and time constants for a BT bifurcation to occur. They propose the BT normal form as a generic minimal model for a single neuron.

Numerically, we applied our analytical results to three examples and compared them with the computations of \textsf{MATCONT}, a numerical bifurcation analysis toolbox in \textsf{Matlab}.   
Furthermore, we constructed  bifurcation  diagrams  using \textsf{MATCONT} to explain the possible behaviour of each example and discuss the switches in the neuronal excitability class with respect to the M-current $g_M$. As predicted by normal form theory~\cite{bognadov1975versal,takens1974singularities,Guck_Holmes,kuznetsov2013elements} in all examples a curve of homoclinic bifurcation, a curve of Hopf bifurcation  and a curve of saddle-node of equilibria emanate from the BT point. These latter two curves particularly affect the
neuronal excitability class. 
We found that a transition is determined by the BT point which occurs at $(g_M,I_{app})=(g_M^*,I_{app}^*)$. The model is a Class-I oscillator when $g_M<g_M^*$ and Class-II  when $g_M>g_M^*$.
More precisely, when $g_M<g_M^*$ as $I_{app}$ is increased oscillations with arbitrarily slow frequency appear via a saddle-node on invariant circle bifurcation  while when $g_M>g_M^*$ oscillations with a positive frequency
appear via a fold bifurcation of cycles, followed by a subcritical Hopf bifurcation.

Using systems of two synaptically coupled cells, we explored how the change in excitability class with the variation of $g_M$ affects synchronization in the example models. We found that while the excitability class of the model changes exactly at the BT point, the synchronization property of all the models switches at $g_M$ value larger than the BT point. We attributed this to the change to the fact that the M-current also affects the frequency of the intrinsic oscillations and that the synchronization of class II oscillators has been shown to be sensitive to intrinsic frequency. Thus the necessary condition for the switch of synchronization, we observed is that the system be class II and the frequency be sufficiently small.

We also considered the effect of the leak conductance, $g_L$, showing that, in the examples we considered, increasing $g_L$ decreases the $g_M$ value of the BT point. This means that the range of values of $g_M$ where the model has class I excitability will be decreased. Equivalently, smaller changes of $g_M$ are needed to switch the model from class I to class II.  If $g_L$ is increased enough then $g_M^*$ may become negative, in which case the model will exhibit class II excitability regardless of the value of $g_M$. Since the switch of synchronization occurs at a higher value of $g_M$ than the BT point, this does not necessarily mean the system will not exhibit changes in synchronization associated with a change in $g_M$, it just means that smaller changes in $g_M$ are needed to switch the synchronization property. We note that Prescott et al.~\cite{prescott2006nonlinear,prescott2008pyramidal} represented the increase in membrane conductance due to background synaptic input using a leak current with a reversal potential near rest in a Morris-Lecar model with an M-current. The one parameter bifurcation diagrams in \cite{prescott2008pyramidal} are consistent with what we have seen in  our analysis.
Our analysis of the effect of $g_L$ on the BT point relies on understanding how the intersection points of two curves vary with $g_L$. Only one curve depends on $g_L$ and we can show in general (i.e., for any model) that the curve will move downward as $g_L$ increases. This effect depends on two aspects of the M-current: the reversal potential is a large negative
value (since it is a potassium current) and the current is noninactivating, see eq. \eqref{gMstar}.

The implications of these results for the action of acetylcholine are as follows. If the neuron is of Class-II in the absence of
acetylcholine (corresponding to high $g_M$) then the presence of
acetylcholine may push the system past the BT bifurcation point
and change the neural excitability type to Class-I. The expected synchronization in the presence of sufficient acetylcholine is then clear: neurons with excitable coupling will likely desynchronize while those with inhibitory coupling will synchronize. This is consistent with the changes to the PRCs induced by acetylcholine
observed in \cite{stiefel2008cholinergic}.
Whether or not acetylcholine induces a change in synchronization may depend on intrinsic firing frequencies of cells. 
Expanding on the idea of Prescott et al.\cite{prescott2006nonlinear,prescott2008pyramidal}, an increase in membrane input conductance would make the system more sensitive to the effects of acetylcholine,
so that switches of synchronization could occur more easily.

These conclusions, of course, assume that the only affect of acetylcholine is to down-regulate the M-current. However, acetycholine has been observed to have other effects, including down-regulating an afterhyperpolarization current $I_{AHP}$ \cite{madison1987voltage,mccormick1993actions} and the leak current \cite{stiefel2008cholinergic}. As indicated above, our work indicates that decreasing $g_L$ will increase the value of $g_M^*$. Thus the simultaneous downregulation of the leak and M-currents would cause the switch of excitability class at higher values of $g_M$. The net effect would be to increase the sensitivity of the model to acetylcholine.
We leave the exploration of the effect of the $I_{AHP}$ current for future work.

The effect of acetylcholine, through the M-current, on the synchronization of cells has been explored using numerical simulations and phase response curves \cite{stiefel2008cholinergic,stiefel2009effects,fink2011cellularly,ermentrout2010mathematical}. We have linked these effects to a particular bifurcation structure of conductance-based models with an M-current and given conditions for this to occur in any conductance-based model. This approach allows us to generalize previous results and to easily explore the effect of multiple parameters in these models.

\section*{Appendix: Parameters, Units, and Functions in Section \ref{sec_Numerical}}

 \label{sec_App}
\begin{itemize}
    \item \textbf{Example 1: Wang--Buz\'saki model.} The infinity and $\tau_{\sigma}$, $\sigma\in\{m,h,n\}$, functions are
	\begin{neweq_non}
		~&{m_\infty }(V) = \frac{{{\alpha_m}(V)}}{{{\alpha_m}(V) + {\beta_m}(V)}},\quad {h_\infty }(V) = \frac{{{\alpha_h}(V)}}{{{\alpha_h}(V) + {\beta_h}(V)}},\\
		~&{n_\infty }(V) = \frac{{{\alpha_n}(V)}}{{{\alpha_n}(V) + {\beta_n}(V)}},\quad {w_\infty }(V) =\frac{1}{{{e^{ - \frac{{V + 27}}{7}}} + 1}} ,\\
	~&{\tau _w}(V) = \frac{1}{0.003\left(e^{\frac{V+63}{15}}+e^{\frac{-(V+63)}{15}}\right)},\quad{\tau _h}(V) = \frac{1}{{{\alpha_h(V)} + {\beta_h(V)}}},\\
	~&	{\tau _n}(V) = \frac{1}{{{\alpha_n(V)} + {\beta_n(V)}}},
	\end{neweq_non}
	where the rate constants $\alpha_{\sigma}$ and $\beta_{\sigma}$ are:
	\begin{neweq_non}
		~&{\alpha _m}(V) =  - \frac{{0.1(V + 35)}}{{{e^{ - 0.1(V + 35)}} - 1}},\quad{\alpha _h}(V) = 0.07{e^{ - \frac{{V + 58}}{{20}}}},\\
		~&{\alpha _n}(V) =  - \frac{{0.01(V + 34)}}{{{e^{ - 0.1(V + 34)}} - 1}},\quad \quad{\beta _n}(V) = 0.125{e^{ - \frac{{V + 44}}{{80}}}},\\
		~&{\beta _m}(V) = 4{e^{ - \frac{{V + 60}}{{18}}}},\quad{\beta _h}(V) = \frac{1}{{{e^{ - 0.1(V + 28)}} + 1}}.
	\end{neweq_non}
	Parameter values are listed in Table \ref{WB_parameters}.

	\begin{table}[htb!]
		\centering
		\begin{tabular}{|c|c|c|c|}
			\hline
			Conductance ($m\mathrm{S}/\mathrm{cm}^2$)& Reversal potential ($m\mathrm{V}$)& Capacitance ($\mu\mathrm{F}/\mathrm{cm}$) & Others  \\ \hline
			$g_L=0.1$      & $V_L=-65$             & $C_M=1$    & $\phi=5$ \\
			$g_{Na}=35$   & $V_{Na}=55$          &             &         \\
			$g_{K}=9$    & $V_K=-90$          &             &         \\ \hline
		\end{tabular}
		\caption{Parameter values for Example 1: Wang--Buz\'saki model (\ref{WB_model}).}
		\label{WB_parameters}
	\end{table}

        \item \textbf{Example 2:  Stiefel model.} The functions are 
\begin{neweq_non}
		~&{m_\infty }(V) = \frac{1}{e^{\frac{-(V+30)}{9.5}}+1},\quad {w_\infty }(V) = \frac{1}{e^{\frac{-(V+39)}{5}}+1},\\
		~& {h_\infty }(V) = \frac{1}{e^{\frac{V+53}{7}}+1},\quad \tau_{w}(V)=75,\\
		~&{n_\infty }(V) = \frac{1}{e^{\frac{-(V+30)}{10}}+1},\quad \tau_{h}(V)=0.37+\frac{2.78}{e^{\frac{V+40.5}{6}}+1},\\ ~&\tau_{n}(V)=0.37+\frac{1.85}{e^{\frac{V+27}{15}}+1}.
	\end{neweq_non}

	\begin{table}[htb!]
		\centering
		\begin{tabular}{|c|c|c|c|}
			\hline
			Conductance ($m\mathrm{S}/\mathrm{cm}^2$)& Reversal potential ($m\mathrm{V}$)& Capacitance ($\mu\mathrm{F}/\mathrm{cm}$) & Others  \\ \hline
			$g_L=0.02$      & $V_L=-60$             & $C_M=1$    & $\phi_w=1$ \\
			$g_{Na}=24$   & $V_{Na}=55$          &             &   $\phi_h=1$      \\
			$g_{K}=3$    &    $V_K=-90$     &             &    $\phi_n=1$     \\  \hline
		\end{tabular}
		\caption{Parameter values for Example 2: Stiefel model.}
		\label{Stiefel_parameters}
	\end{table}

        \item \textbf{Example 3: Reduced Traub-Miles model.} 
	The infinity and $\tau_{\sigma}$, $\sigma\in\{m,h,n,w\}$, 
	functions are
	\begin{neweq_non}
		~&{m_\infty }(V) = \frac{{{\alpha_m}(V)}}{{{\alpha_m}(V) + {\beta_m}(V)}},\quad {h_\infty }(V) = \frac{{{\alpha_h}(V)}}{{{\alpha_h}(V) + {\beta_h}(V)}},\\
		~&{n_\infty }(V) = \frac{{{\alpha _n}(V)}}{{{\alpha _n}(V) + {\beta _n}(V)}},\quad {w_\infty }(V) = \frac{1}{{{e^{\frac{{ - \left( {V + 35} \right)}}{{10}}}} + 1}},\\
		~& {\tau _w}(V) = \frac{{400}}{{3.3{e^{\frac{{V + 35}}{{20}}}} + {e^{\frac{{ - (V + 35)}}{{20}}}}}},\quad {\tau _n}(V) = \frac{1}{{{\alpha_n(V)} + {\beta_n(V)}}},\\
		~&{\tau _m}(V) = \frac{1}{{{\alpha_m(V)} + {\beta_m(V)}}}\quad{\tau _h}(V) = \frac{1}{{{\alpha_h(V)} + {\beta_h(V)}}},
	\end{neweq_non}
	where the rate constants $\alpha_{\sigma}$ and $\beta_{\sigma}$ are:
	\begin{neweq_non}
		~&{\alpha _m}(V) = \frac{{0.32(V + 54)}}{{1 - {e^{ - \frac{{V + 54}}{4}}}}},\quad {\alpha _h}(V) = 0.128{e^{ - \frac{{V + 50}}{{18}}}},\\
		~& {\alpha _n}(V) = \frac{{0.032(V + 52)}}{{1 - {e^{ - \frac{{V + 52}}{5}}}}}\quad {\beta _n}(V) = 0.5{e^{ - \frac{{V + 5}}{{40}}}},\\
		~&{\beta _m}(V) = \frac{{0.28(V + 27)}}{{{e^{\frac{{V + 27}}{5}}} - 1}},\quad {\beta _h}(V) = \frac{4}{{{e^{ - \frac{{V + 27}}{5}}} + 1}}.
	\end{neweq_non}

	\begin{table}[htb!]
		\centering
		\begin{tabular}{|c|c|c|}
			\hline
			Conductance ($m\mathrm{S}/\mathrm{cm}^2$)& Reversal potential ($m\mathrm{V}$)& Capacitance ($\mu\mathrm{F}/\mathrm{cm}$)   \\ \hline
			$g_L=0.1$      & $V_L=-67$             & $C_M=1$    \\
			$g_{Na}=100$   & $V_{Na}=50$          &                     \\
			$g_{K}=80$    & $V_K=-100$          &                     \\ \hline
		\end{tabular}
		\caption{Parameter values for Example 3: RTM model (\ref{RTM_model}).}
		\label{RTM_parameters}
	\end{table}

\end{itemize}

\bibliographystyle{ieeetr}
\bibliography{references}

\end{document}